\documentclass[11pt]{article}
\usepackage{amssymb}
\usepackage{amsfonts}
\usepackage{amsmath}
\usepackage{amscd,amsmath,amsfonts,amssymb,latexsym,graphicx,url, color,cite,comment,upgreek,esint,xspace}
\newtheorem{theorem}{Theorem}[section]

\newtheorem{lemma}[theorem]{Lemma}

\newtheorem{proposition}[theorem]{Proposition}
\newtheorem{remark}[theorem]{Remark}

\newenvironment{proof}[1][Proof]{\textbf{#1.} }{\hfill\rule{0.5em}{0.5em}}
{\catcode`\@=11\global\let\AddToReset=\@addtoreset
\AddToReset{equation}{section}

\AddToReset{theorem}{section}

\begin{document}

\title{Liouville results and asymptotics of solutions of a quasilinear elliptic
equation with supercritical  source gradient term}
\author{Marie-Fran\c{c}oise Bidaut-V\'{e}ron\footnote{\noindent
Laboratoire de Math\'{e}matiques et Physique Th\'{e}orique, Universit\'e de Tours, 37200 Tours, France. E-mail: veronmf@univ-tours.fr}}
\maketitle

\begin{abstract}
We consider the elliptic quasilinear equation $-\Delta _{m}u=u^{p}\left\vert
\nabla u\right\vert ^{q}$ in $\mathbb{R}^{N}$ with $q\geq m$ and $p>0,$ $ 
1<m<N.$ Our main result is a Liouville-type property, namely, all the
positive $C^{1}$ solutions in $\mathbb{R}^{N}$ are constant. We also give
their asymptotic behaviour : all the solutions in an exterior domain $ 
\mathbb{R}^{N}\backslash B_{r_{0}}$ are bounded. The solutions in $%
B_{r_{0}}\backslash \left\{ 0\right\} $ can be extended as a continuous
functions in $B_{r_{0}}.$ The solutions in $\mathbb{R}^{N}\backslash \left\{
0\right\} $ has a finite limit $l\geq 0$ as $\left\vert x\right\vert
\rightarrow \infty $. Our main argument is a Bernstein estimate of the
gradient of a power of the solution, combined with a precise Osserman's type
estimate for the equation satisfied by the gradient.
\end{abstract}

  \vspace{1mm}

\tableofcontents
\hspace{.05in}

\section{ Introduction}

In this paper we study local and global properties of positive
solutions of the equation 
\begin{equation}
-div\left(|\nabla u|^{m-2}\nabla u\right):=-\Delta _{m}u=u^{p}\left\vert \nabla u\right\vert ^{q},  \label{eqm}
\end{equation}%
in $\mathbb{R}^{N},$ ($N\geq 1$, $1<m<N$ and $p>0$) in the
supercritical case 
\begin{equation}
q\geq m.  \label{sup}
\end{equation}%
We are concerned by the Liouville property in $\mathbb{R}^{N},$ which
 is wether all the positive $C^{1}$ solutions are constant. We
also study the asymptotic behaviour of any solution of (\ref{eqm}) near a
singularity in the punctured ball $B_{r_{0}}\backslash \left\{ 0\right\} $, in $\mathbb{R}%
^{N}\backslash \left\{ 0\right\} $ or in an exterior domain $\mathbb{R}%
^{N}\backslash B_{r_{0}}.$

In the case $q=0,$ equation (\ref{eqm}) reduces to the 
classical Lane-Emden-Fowler equation 
\begin{equation}
-\Delta _{m}u=u^{p},  \label{EFm}
\end{equation}%
which has already been the subject of countless publications. One of the questions solved is 
that the Liouville property holds if and only if 
\begin{equation*}
p<p_{m}^{\ast }:=\frac{N(m-1)+m}{N-m}.
\end{equation*}%
Note that $p_{m}^{\ast }$ is the Sobolev exponent. Since it is impossible to quote all the articles on
the subject, we only mention here the pioneering works and references therein.
Gidas and Spruck \cite{GiSp} first showed the nonexistence of positive
solutions in $\mathbb{R}^{N}$ for $m=2$ and $p<p_{2}^{\ast }$. They combine the
Bernstein technique applied in the equation satisfed by the gradient of a
suitable power of $u$, with delicate integral estimates ensuring the
Harnack inequality, see also \cite{BiVe}. Then the complete behaviour up to
the case $p=p_{2}^{\ast }$ was obtained by moving plane methods by \cite{CaGiSp},
see also \cite{ChLi}. In the general case $m>1,$ the nonexistence of
nontrivial solutions for $p<p_{2}^{\ast }$ was proved in a beautiful article of
Serrin and Zou \cite{SeZo}, then the extension to the case $p=p_{2}^{\ast }$ was
done by \cite{ScMeMo} for $m<2,$ then \cite{Vet} for $1<m<2$, and finally 
\cite{Sc} for any $m>1.$

When $p=0$, (\ref{eqm}) reduces to the Hamilton-Jacobi equation 
\begin{equation*}
-\Delta _{m}u=\left\vert \nabla u\right\vert ^{q}.
\end{equation*}%
The Liouville property was proved in \cite{Lio} for $m=2,$ and in \cite%
{BiGaYa} for any $m>1,$ using the Bernstein technique. In that case the
nonexistence holds for any $q>m-1,$ without any sign condition on the
solution. Estimates of the gradient for more general problems can be found
in \cite{LePo}.\smallskip

For the general case of equation (\ref{eqm}), consider the range of exponents
\begin{equation*}
p>0,\qquad p+q+1-m>0.
\end{equation*}
As in the case $q=0,$ there exists a "first subcritical case", where 
\begin{equation*}
p<\frac{N(m-1)}{N-m}-\frac{(N-1)q}{N-m},
\end{equation*}%
for which any supersolution in $\mathbb{R}^{N}$ of equation (\ref{eqm}) is
constant, from \cite{Fi}. Beyond this case, a second critical case appears
when $0\leq q<m-1$: indeed there exist radial positive nonconstant solutions
of (\ref{eq2}) whenever $p\geq p_{m,q}^{\ast },$ where 
\begin{equation*}
p_{m,q}^{\ast }=\frac{N(m-1)+m}{N-m}-\frac{q((N-1)q-N(m-1)+m)}{(N-m)(m+1-q)},
\end{equation*}%
see \cite{CaMi} and \cite{BiGaVe2}.

When $m=2<N$ and $p>0,$ equation

\begin{equation}
-\Delta u=u^{p}\left\vert \nabla u\right\vert ^{q}  \label{eq2}
\end{equation}%
was studied in \cite{BiGaVe2} for $0<q\leq 2.$ The case $q=2$ could be
solved explicitely by a change of the unknown function, showing that the Liouville property holds for any $p>0.$
Using a direct Bernstein technique we obtained a first range of values of $%
(p,q)$ for which the Liouville property holds, in particular it holds when $%
p+q-1<\frac{4}{N-1},$ covering the first subcritical case. Using an
integral Bernstein technique in the spirit of \cite{GiSp} we obtained a
wider range of $(p,q)$ ensuring the Liouville property, recovering Gidas and
Spruck result $p<\frac{N+2}{N-2}$ when $q=0.$ However some deep questions remained unsolved: 
Does the property hold for any $p<p_{2,q}^{\ast }$ when $q<1$ ?
Does it hold for any $p>0$ when $1\leq q<2$ ?\smallskip

In a recent article, Filippucci, Pucci and Souplet \cite[Theorem 1.1]{FiPuSo} considered
the case $m=2,$ $q>2,$ of a superquadratic growth in the gradient, a case which was
not covered by \cite{BiGaVe2}. They proved the following:\medskip

\noindent{\bf Theorem}\cite[Theorem 1.1]{FiPuSo}
{\it Any classical positive and \textbf{bounded }solution of
equation (\ref{eq2}) in $\mathbb{R}^{N}$ with $q\geq 2$ and $p>0$ is constant}.
\medskip

In this article, we prove that the Liouville property holds true not only for (\ref{eq2}) but for the
quasilinear equation (\ref{eqm})  \textbf{without the
assumption of boundedness} on the solution. Our main result is the following

\begin{theorem}
\label{main}Let $u$ be any positive $C^{1}(\mathbb{R}^{N})$ solution of
equation (\ref{eqm}), with $1<m<N$ and  
\begin{equation}
q\geq m,\qquad p\geq 0. \label{hyp}
\end{equation}%
Then $u$ is constant.
\end{theorem}

We show that the case $q=m$ can still be solved explicitely, giving the
complete behaviour of the solutions of the equation, see Theorem \ref{crit}.
Next we assume $q>m.$ We prove that all the solutions in an exterior domain
are \textbf{bounded}, and we give the asymptotic behaviour ($|x|\to 0$ and $|x|\to \infty$) of the solutions in $%
\mathbb{R}^{N}\backslash \left\{ 0\right\} $:

\begin{theorem}
\label{exte}Assume $1<m<N,$ $q>m,$ $p\geq 0$. Then any positive $C^{1}$ solution $%
u$ of (\ref{eqm}) in $\mathbb{R}^{N}\backslash B_{r_{0}}$ is bounded. If $u$
is a non-constant solution, then $\left\vert \nabla u\right\vert $ does not vanish for $%
\left\vert x\right\vert >r_{0}$. Moreover any positive solution $u$ in $\mathbb{R}%
^{N}\backslash \left\{ 0\right\} $ satisfies 
\begin{equation}
\lim_{\left\vert x\right\vert \rightarrow \infty }u(x)=l\geq 0.  \label{liml}
\end{equation}%
If $l>0,$ there exist constants $C_{1},C_{2}>0$ such that for $\left\vert
x\right\vert $ large enough, 
\begin{equation}
C_{1}\left\vert x\right\vert ^{\frac{N-m}{m-1}}\leq \left\vert
u(x)-l\right\vert \leq C_{2}\left\vert x\right\vert ^{\frac{N-m}{m-1}}.
\label{enca}
\end{equation}
\end{theorem}

Concerning the solutions in $B_{r_{0}}\backslash \left\{ 0\right\} $ and in particular
 in $\mathbb{R}^{N}\backslash \left\{ 0\right\} $ we proved an estimate
of the gradient, showing that the solution is continuous up to $0$ but the gradient is singular at $0$:

\begin{theorem}
\label{ball}Assume $1<m<N,$ $q>m,$ $p\geq 0$. Any positive solution $u$ in $%
B_{r_{0}}\backslash \left\{ 0\right\} $ is bounded near $0,$ it can be
extended as a continuous function in $B_{r_{0}},$ such that $u(0)>0,$ and
for any $x\in B_{\frac{r_{0}}{2}}\backslash \left\{ 0\right\} $%
\begin{equation}
\left\vert \nabla u(x)\right\vert \leq C\left\vert x\right\vert ^{-\frac{1}{%
q-m+1}},  \label{inv}
\end{equation}%
where $C=C(N,p,q,m,u)$. Finally 
\begin{equation}
\left\vert u(x)-u(0)\right\vert \leq C\left\vert x\right\vert ^{\frac{q-m}{%
q-m+1}},  \label{tur}
\end{equation}%
near $0,$ where $C=C(N,p,q,m,u(0)).$ Moreover, if $u$ is defined in $\mathbb{R%
}^{N}\backslash \left\{ 0\right\} ,$ then $u(x)\leq u(0)$ in $\mathbb{R}%
^{N}\backslash \left\{ 0\right\} .$
\end{theorem}

Note that the exponent involved in (\ref{inv}) is independent of $p$, actualy the
solution behaves like a solution of the Hamilton-Jacobi equation 
\begin{equation}
-\Delta _{m}u=c\left\vert \nabla u\right\vert ^{q},  \label{HJ}
\end{equation}%
with $c=u^{p}(0).$\smallskip

Finally we make an exhaustive study of the radial solutions for $q>m$,
showing the sharpness of the nonradial results. We reduce the study to the one of an autonomous quadratic 
polynomial  system of order 2, following the
technique introduced in \cite{BiGi}. Compared to other classical techniques,
it provides a complete description of all the positive solutions, in particular
the global ones, without questions of regularity. We prove the following:

\begin{theorem}
\label{radial} Assume $1<m<N,$ $q>m,$ $p\geq 0$ and $u$ is any positive non constant radial solution $r\mapsto u(r)$ of (%
\ref{eqm})  in an interval $(a,b)\subseteq (0,\infty )$.\smallskip

\noindent(i) If $a=0$, then $u$ is bounded, decreasing and
singular: 
\begin{equation}
\lim_{r\rightarrow 0}u=u_{0}>0,\qquad \lim_{r\rightarrow 0}r\left\vert
u^{\prime }\right\vert ^{q-m+1}=\frac{a_{m,q}}{u_{0}^{p}},\qquad a_{m,q}=%
\frac{(N-1)q-N(m-1)}{q+1-m}.  \label{ima}
\end{equation}%
And for given $u_{0}>0,$ there exist infinitely many such solutions;\smallskip

\noindent(ii) If $b=\infty$, then $u$ admits a limit a limit $l\geq 0$ at infinity and 
\begin{equation}
\lim_{r\rightarrow \infty }r^{\frac{N-m}{m-1}}\left\vert u(r)-l\right\vert
=k>0.  \label{ime}
\end{equation}%
Furthermore, for given $l>0,$ $c\neq 0$ there exists a unique local solution near $%
\infty $, such that 
\begin{equation}
\lim_{r\rightarrow \infty }r^{\frac{N-m}{m-1}}(u(r)-l)=c.  \label{imo}
\end{equation}\smallskip

\noindent(iii) For any $u_{0}>0,$ there exist infinitely many solutions in $(0,\infty
),$ decreasing, such that $\lim_{r\rightarrow 0}u=u_{0},$ but  a unique one, satisfying 
\begin{equation}
\lim_{r\rightarrow 0}u=u_{0}\text{ \quad and \quad }\lim_{r\rightarrow
\infty }u=0.  \label{imu}
\end{equation}%
There exist infinitely many solutions defined on an interval $(0,\rho ),$
such that $\lim_{r\rightarrow \rho }u=0$, and an infinity such that $%
\lim_{r\rightarrow \rho }u^{\prime }=-\infty .$ Finally, there exist an infinity of
solutions in $(\rho ,\infty )$ such that $\lim_{r\rightarrow \rho }u=0,$ and
an infinity of solutions such that $\lim_{r\rightarrow \rho }u^{\prime
}=\infty .$
\end{theorem}

Note that Theorems \ref{exte} and \ref{ball} lead to the following natural question: are
all the solutions in $\mathbb{R}^{N}\backslash \left\{ 0\right\} $ radially
symmetric? This is still an open problem, even in the case $p=0$ of the
Hamilton-Jacobi equation.

To conclude this paper, we improve another result of \cite{FiPuSo}, where it
was noticed that \cite[Theorem 1.1]{FiPuSo} was still valid for $p<0,$ $q\geq 2.$
We prove here a much more general result covering the case $%
p=0$.

\begin{theorem}
\label{pneg}  Assume $1<m<N,$ $p\leq 0$ and \ $p+q+1-m>0.$ Then there exists
a constant $C=C(N,p,q,m)>0$ such that for any positive $C^{1}$ solution $u$
of (\ref{eqm}) in a bounded domain $\Omega ,$ 
\begin{equation*}
\left\vert \nabla u(x)\right\vert \leq C\text{ dist}(x,\partial \Omega )^{-%
\frac{1}{q+1-m}},\quad \forall x\in \Omega .
\end{equation*}%
If $\Omega =\mathbb{R}^{N},$ then $u$ is constant.\medskip
\end{theorem}

Let us give a brief comment on the analogous equation with an absorption
term:%
\begin{equation}
-\Delta _{m}u+u^{p}\left\vert \nabla u\right\vert ^{q}=0 . \label{eqa}
\end{equation}%
In
the case $m=2$, $0<q<2$, a complete classification of the solutions with isolated singularities 
was performed in \cite{ChCi1}. A main contribution
was recently given by the same authors in \cite{ChCi2} where they obtained
optimal estimates of the gradient for any $1<m\leq N$, $p,q\geq 0,$ $%
p+q-m+1>0,$ still by the Bernstein method.\medskip \medskip

Our paper is organized as follows. In Section \ref{qm} we first treat the
case $q=m.$ In Section \ref{ideas} we give the main ideas of our proofs 
when $q>m=2$, and we introduce some tools for the general case $q>m>1$. 
Our main theorems are proved in Section 4,  and %
Section 5 is devoted to the radial case. The
extension to the case $p\leq 0$ is given in Section 6.

\section{The case $q=m$\label{qm}}

If $q=m$ we can express explicitely the solutions of (\ref{eqm}). We prove the following:

\begin{theorem}\label{crit}
Let $1<m<N,$ $p\geq 0,q=m.$ Then\smallskip

\noindent (i) any $C^{1}$ positive solution in $\mathbb{R}^{N}$ is constant;\smallskip

\noindent (ii) any nonconstant positive solution in $\mathbb{R}^{N}\backslash
B_{r_{0}} $ has a limit $l$ at $\infty $ and 
\begin{equation*}
\lim_{\left\vert x\right\vert \rightarrow \infty }\left\vert x\right\vert ^{%
\frac{m-N}{m-1}}\left\vert u-l\right\vert =c>0;
\end{equation*}\smallskip

\noindent (iii) any positive solution in $B_{r_{0}}\backslash \left\{ 0\right\} $
extends as a continuous function in $B_{r_{0}},$ or satisfies 
\begin{equation}
\lim_{x\rightarrow 0}\frac{u^{p+1}}{\left\vert \ln \left\vert x\right\vert
\right\vert }=\frac{(N-m)(p+1)}{m-1};  \label{mil}
\end{equation}\smallskip

\noindent (iv) any positive solution in $\mathbb{R}^{N}\backslash \left\{ 0\right\} $ is radial.
\end{theorem}

\noindent\begin{proof}
We use a change of variable already considered in \cite{AbBi}: the equation
takes the form 
\begin{equation}
{-\Delta }_{m}{u}=\beta (u)\left\vert \nabla u\right\vert ^{m},\text{ with }%
\beta (u)=u^{p}.  \label{eqo}
\end{equation}%
We set $\gamma (\tau )=\int_{0}^{\tau }\beta (\theta )d\theta =\frac{%
\tau ^{p+1}}{p+1},$ and 
\begin{equation}
U(x)=\Psi (u(x))=\int_{0}^{u(x)}e^{\frac{\gamma (\theta )}{m-1}}d\theta
:=\int_{0}^{u(x)}e^{\frac{\theta ^{p+1}}{(p+1)(m-1)}}d\theta.  \label{mach}
\end{equation}%
A function $u$ is a  solution of (\ref{eqm}) if and only if the above function $U$ 
 satisfies
\begin{equation*}
{-\Delta _{m}U}=0,
\end{equation*}%
and if $u$ is nonnegative not identically $0$,  $U$ is $m$-harmonic and
positive. Conversely, $u$ is derived from $U$ by 
\begin{equation}
u(x)=\Psi ^{-1}(U(x))=\int_{0}^{U(x)}\frac{ds}{1+g(s)}\,\, \text{ where }\;%
g(s)=\int_{0}^{s}\beta (\Psi (w))dw=\int_{0}^{s}\Psi ^{p}(w)dw.  \label{jil}
\end{equation}%
(i) If $u$ is a solution in $\mathbb{R}^{N}$ of (\ref{eqo}), it is
constant. Indeed any nonnegative $m$-harmonic functions $U$ defined in $%
\mathbb{R}^{N}$ is constant, from the Harnack inequality, see \cite{Re}, 
\cite{Se} and \cite[Theorem II]{SeZo}.

\noindent (ii) If $u$ is defined in $\mathbb{R}^{N}\backslash B_{r_{0}},$ then $U 
$ is bounded, it admits a limit $L$ at $\infty $ and there holds 
$\left\vert U(x)-L\right\vert
\leq C\left\vert x\right\vert ^{\frac{p-N}{p-1}}$ near $\infty ,$ see \cite%
{AvBr} for more general results. Clearly the same properties hold for $u$ (with another limit).

\noindent (iii) If $u$ is defined in $B_{r_{0}}\backslash \left\{ 0\right\}$, it follows from \cite{Se} that, either 
$U$ extends as a continuous $m$-harmonic function
in $B_{r_{0}},$ or it behaves like $k\left\vert x\right\vert ^{\frac{p-N}{p-1%
}}$ near $0,$ so (\ref{mil}) holds.

\noindent (iv) If $u$ is a solution in $\mathbb{R}^{N}\backslash \left\{ 0\right\}$, 
it is proved in \cite{KiVe} that $U$ is radial and endows the form
\begin{equation*}
U(x)=k\left\vert x\right\vert ^{\frac{m-N}{m-1}}+\lambda \text{ \qquad with }%
k,\lambda \geq 0.
\end{equation*}%
Then $u$ is radial, and, using (\ref{jil}), it has the expression
\begin{equation*}
u(x)=\int_{0}^{\lambda }\frac{ds}{1+g(s)}+\int_{0}^{k\left\vert x\right\vert ^{%
\frac{m-N}{m-1}}}\frac{ds}{1+g(s-\lambda )}.
\end{equation*}
\end{proof}

\section{Main arguments of the proofs\label{ideas}}

\subsection{Ideas in the case $m=2$}
Before detailling the  proof of Theorem \ref{main}. for $q>m,$ we give an
overview of it in the simple case of equation (\ref{eq2}), with $m=2$, $%
p>0,q>2.$ We set $u=v^{b},$ with $b\in (0,1),$ and obtain 
\begin{equation*}
-\Delta v=(b-1)\frac{\left\vert \nabla v\right\vert ^{m}}{v}%
+b^{q-1}v^{s}\left\vert \nabla v\right\vert ^{q},
\end{equation*}%
with $s=1-q+b(p+q-1).$ Next we explicit the equation satisfied by $%
z=\left\vert \nabla v\right\vert ^{2}.$ Taking in account the B\"{o}chner
formula and Cauchy-Schwarz inequality  in $\mathbb{R}^{N}$, 
\begin{equation*}
-\frac{1}{2}\Delta z+\frac{1}{N}(\Delta v)^{2}+<\nabla (\Delta v),\nabla
v)\leq -\frac{1}{2}\Delta z+(Hessv)^{2}+<\nabla (\Delta v),\nabla v>=0,
\end{equation*}%
we get an estimate of the form, with universal constants $C_{i}>0,$ 
\begin{equation}
-\Delta z+C_{1}v^{2s}z^{q}\leq C_{2}\frac{z^{2}}{v^{2}}+C_{3}\frac{1}{v}%
<\nabla z,\nabla v>+C_{4}v^{s}z^{\frac{q-2}{2}}<\nabla z,\nabla v>,  \notag
\end{equation}%
then
\begin{equation}
-\Delta z+C_{5}v^{2s}z^{q}\leq C_{6}\frac{z^{2}}{v^{2}}+C_{7}\frac{%
\left\vert \nabla z\right\vert ^{2}}{z}.  \label{thu}
\end{equation}%
Using the H\"{o}lder inequality we deduce, 
\begin{equation*}
-\Delta z+C_{8}v^{2s}z^{q}\leq C_{9}v^{-\frac{2(q+2s)}{q-2}}+C_{7}\frac{%
\left\vert \nabla z\right\vert ^{2}}{z}.
\end{equation*}%
The crucial step is an estimate of Osserman's type in a ball $%
B_{\rho }$ valid for functions satisfying the inequality%
\begin{equation*}
-\Delta z+\alpha (x)z^{k}\leq \beta (x)+d\frac{\left\vert \nabla
z\right\vert ^{2}}{z}\qquad \text{in }B_{\rho },
\end{equation*}%
where $k>1$. This is proved in Lemma \ref{oss} below, and it asserts that 
\begin{equation*}
z(x)\leq C(N,k,d)\left(\frac{1}{\rho ^{2}}\max_{B_{\rho }}\frac{1}{\alpha }\right)^{%
\frac{1}{k-1}}+\left(\max_{B_{\rho }}\frac{\beta }{\alpha }\right)^{\frac{1}{k}}\qquad 
\text{in }B_{\frac{\rho }{2}}.
\end{equation*}%
Then we take $b=\frac{q-2}{p+q-1},$ in the same spirit as in \cite{FiPuSo}, so
that $\frac{B}{\alpha }$ is constant and $\alpha ^{-1}(x)=v^{2}(x)$. We
obtain an estimate 
\begin{equation*}
\max_{\bar{B}_{\frac{\rho }{2}}}\left\vert \nabla v\right\vert \leq C\left(\left(\frac{%
\max_{B_{\rho }}v}{\rho }\right)^{\frac{1}{q-1}}+1\right).
\end{equation*}%
But any solution in $\mathbb{R}^{N}$ satisfies for any $\rho \geq 1$ 
\begin{equation}
\max_{B_{\rho }}v\leq v(0)+C\rho \max_{B_{\rho }}\left\vert \nabla
v\right\vert \leq C\rho (1+\max_{B_{\rho }}\left\vert \nabla v\right\vert ),
\label{vzero}
\end{equation}%
which yields 
\begin{equation*}
\max_{B_{\frac{\rho }{2}}}\left\vert \nabla v\right\vert \leq
C((\max_{B_{\rho }}\left\vert \nabla v\right\vert )^{\frac{1}{q-1}}+1).
\end{equation*}%
Using the bootstrap method developped in \cite{BiGr1} and \cite%
{BiGaYa} based upon the fact that $\frac{1}{q-1}<1,$ we deduce that $\left\vert \nabla v\right\vert $ $\in L^{\infty }(%
\mathbb{R}^{N}).$ Note that the boundness of 
$|\nabla v|$ had been obtained in \cite{FiPuSo} but under the extra assumption $u\in L^{\infty
}(\mathbb{R}^{N}),$ an assumption that we get rid of. Returning to $%
u=v^{b},$ it means that 
\begin{equation*}
-\Delta u=u^{p}\left\vert \nabla u\right\vert ^{q}\leq C\frac{\left\vert
\nabla u\right\vert ^{2}}{u},
\end{equation*}%
and the same happens for $u-l$ , where $l=\inf_{\mathbb{R}^{N}}u.$ It
implies that $w_{l}=(u-l)^{\sigma }$ is subharmonic for $\sigma $ large
enough. Then from \cite{BiGaYa}, see also Lemma \ref{subha} below, and since $u$ is
superharmonic, 
\begin{equation*}
\sup_{B_{R}}w_{l}\leq C\left(\frac{1}{|B_{2R}|}\int_{B_{2R}}w^{\frac{1}{\sigma }%
}\right)^{\sigma }=C\left(\frac{1}{|B_{2R}|}\int_{B_{2R}}(u-l)\right)^{\sigma }\leq C^{\prime
}(\inf_{B_{R}}u-l)^{\sigma }.
\end{equation*}%
Since $C'$ is independent of $R$, it follows that $\sup_{\mathbb{R}^{N}}w_{l}=0$, thus $u\equiv l.$

Next we consider a solution in an exterior domain and we replace 
 (\ref{vzero}) by a more precise comparison estimate between $v(x)$
and its infimum on a sphere of radius $\left\vert x\right\vert ,$ and use
the fact that this infimum is bounded as $r\rightarrow \infty .$ Then we can
show that $u$ is still bounded, and obtain the behaviour near $\infty $ by a
careful study of $u$ and $w_{l}.$ Finally we study the behaviour in $%
B_{r_0}\backslash \left\{ 0\right\} $ by the Bernstein technique, not relative
to $v$ but directly to $u,$ that means we take $b=1,$ so that $s=p$. From (\ref{thu})  the
function $\xi =\left\vert \nabla u\right\vert ^{2}$ satisfies 
\begin{equation*}
-\Delta \xi +C_{5}u^{2p}\xi ^{q}\leq C_{6}\frac{\xi ^{2}}{u^{2}}+C_{7}\frac{%
\left\vert \nabla z\right\vert ^{2}}{z},
\end{equation*}%
and $k=\inf_{B_{\frac {r_0}2}\backslash \left\{ 0\right\} }u$ is positive by the
strong maximum principle, thus 
\begin{equation*}
-\Delta \xi +C_{8}\xi ^{q}\leq C_{9}\xi ^{2}+C_{7}\frac{\left\vert \nabla
z\right\vert ^{2}}{z}\leq \frac{C_{8}}{2}\xi ^{q}+C_{11}+C_{7}\frac{%
\left\vert \nabla z\right\vert ^{2}}{z},
\end{equation*}%
from what we deduce the estimates of $\xi .$

\subsection{Some tools\label{tools}}

In the sequel we use the Bernstein method. In the case $p=0,$ it appeared
that the square of the gradient is a subsolution of an elliptic equation
with absorption, for which one can find estimates from above of Osserman's type.
In the case of equation (\ref{eqm}), the problem is more difficult, but such
upper estimates were also a main step in study of \cite{BiGaVe2} of equation
(\ref{eq2}) for $q<2.$ Here also they constitue a crucial step of our proofs
below. The following Lemma gives an Osserman's type property of such
equations, extending  of \cite
[Lemma 2.2]{BiGaVe2}, see also used in \cite[Proposition 2.1]{BiGaVe}.

\begin{lemma}
\label{oss}Let $\Omega $ be a domain of $\mathbb{R}^{N},$ and $z\in C(\Omega
)\cap C^{2}(G),$ where $G=\left\{ x\in \Omega :z(x)\neq 0\right\} .$ Let $%
w\mapsto \mathcal{A}w\mathcal{=}-\sum_{i,j=1}^{N}a_{ij}\frac{%
\partial ^{2}w}{\partial x_{i}\partial x_{j}}$ be a uniformly elliptic
operator in the open set $G:$ 
\begin{equation}
\theta \left\vert \xi \right\vert ^{2}\leq \sum_{i,j=1}^{N}a_{ij}\xi
_{i}\xi _{j}\leq \Theta \left\vert \xi \right\vert ^{2},\qquad \theta >0.
\label{kli}
\end{equation}%
Suppose that for any $x\subset G,$ 
\begin{equation*}
\mathcal{A}(z)+\alpha (x)z^{k}\leq \beta (x)+d\frac{\left\vert \nabla
z\right\vert ^{2}}{z},
\end{equation*}%
with $k>1,$ and $d=d(N,p,q),$ and $\alpha ,\beta $ are continuous in $\Omega 
$ and $\alpha $ is positive. Then there exists $c=c(N,p,q,k)>0$ such that for any ball $\overline{B}(x_{0},\rho
)\subset \Omega $  there holds
\begin{equation*}
z(x_{0})\leq c\left(\frac{1}{\rho ^{2}}\max_{B_\rho (x_{0})}\frac{1}{\alpha }\right)^{%
\frac{1}{k-1}}+\left(\max_{B_\rho (x_{0})}\frac{\beta }{\alpha }\right)^{\frac{1}{k}}.
\end{equation*}
\end{lemma}

\noindent\begin{proof}
Let $\overline{B}(x_{0},\rho )\subset \Omega .$ We can assume that $%
z(x_{0})\neq 0.$ Let $r=\left\vert x-x_{0}\right\vert .$ Let $w$ be the function defined in $B_\rho (x_{0})$  by 
\begin{equation*}
w(x)=\lambda (\rho ^{2}-r^{2})^{-\frac{2}{k-1}}+\mu,
\end{equation*}%
where $\lambda ,\mu >0.$ Let $G_{1}$ be a connected component of $%
\left\{ x\in B_\rho (x_{0});z(x)>w(x)\right\} .$ Then $G_{1}\subset G$ and $%
\overline{G_{1}}\subset \overline{B}(x_{0},\rho )\subset G.$ We define $\mathcal{L}w$
in $B_\rho (x_{0})$ by
\begin{equation*}
\mathcal{L}(w)=\mathcal{A}(w)+\alpha (x)w^{k}-\beta (x)-d\frac{\left\vert
\nabla w\right\vert ^{2}}{w}.
\end{equation*}%
Then
\begin{equation*}
w_{x_{i}}=\frac{4\lambda }{k-1}(\rho ^{2}-r^{2})^{-\frac{2}{k-1}-1}x_{i},
\end{equation*}%
\begin{equation*}
w_{x_{i}x_{j}}=\frac{4\lambda }{k-1}(\rho ^{2}-r^{2})^{-\frac{2}{k-1}%
-1}\delta _{ij}+\frac{4\lambda (k+1)}{(k-1)^{2}}(\rho ^{2}-r^{2})^{-\frac{2}{%
k-1}-2}x_{i}x_{j},
\end{equation*}%
\begin{eqnarray*}
\mathcal{A}(w) &=&-\sum_{i,j=1}^{N}a_{ij}w_{x_{i}x_{j}}=\frac{%
4\lambda }{k-1}(\rho ^{2}-r^{2})^{-\frac{2}{k-1}-1}(-\sum%
_{i,j=1}^{N}a_{ij}\delta _{ij}) \\
&&+\frac{4\lambda (k+1)}{(k-1)^{2}}(\rho ^{2}-r^{2})^{-\frac{2}{k-1}%
-2}(-\sum_{i,j=1}^{N}a_{ij}x_{i}x_{j}) \\
&\geq &-\Theta (\frac{4\lambda N}{k-1}(\rho ^{2}-r^{2})^{-\frac{2}{k-1}-1}+%
\frac{4\lambda (k+1)}{(k-1)^{2}}(\rho ^{2}-r^{2})^{-\frac{2}{k-1}-2}r^{2} \\
&=&-\Theta (\frac{4\lambda N}{k-1}(\rho ^{2}-r^{2})^{-\frac{2}{k-1}%
-2}(N(\rho ^{2}-r^{2})+\frac{k+1}{k-1}r^{2}) \\
&=&-\Theta (\frac{4\lambda }{k-1}(\rho ^{2}-r^{2})^{-\frac{2}{k-1}-2}(N\rho
^{2}+(\frac{k+1}{k-1}-N)r^{2}),
\end{eqnarray*}%
\begin{equation*}
\left\vert \nabla w\right\vert ^{2}=\frac{16\lambda ^{2}}{(k-1)^{2}}(\rho
^{2}-r^{2})^{-\frac{4}{k-1}-2}r^{2}\Longrightarrow \frac{\left\vert \nabla w\right\vert ^{2}}{w}\leq \frac{16\lambda }{(k-1)^{2}%
}(\rho ^{2}-r^{2})^{-\frac{2}{k-1}-2}r^{2},
\end{equation*}%
and
\begin{equation*}
w^{k}=(\lambda (\rho ^{2}-r^{2})^{-\frac{2}{k-1}}+\mu )^{k}\geq \mu
^{k}+\lambda ^{k}(\rho ^{2}-r^{2})^{-\frac{2k}{k-1}}=\mu ^{k}+\lambda
^{k}(\rho ^{2}-r^{2})^{-\frac{2}{k-1}-2}.
\end{equation*}%
We deduce from this series of inequalities,
\begin{eqnarray*}
\mathcal{L}(w) &\geq &\alpha (x)\mu ^{k}-\beta (x)+\lambda (\rho
^{2}-r^{2})^{-\frac{2k}{k-1}}(\lambda ^{k-1}C(x) \\
&&-\Theta (\frac{4}{k-1}(N\rho ^{2}+(\frac{k+1}{k-1}-N)r^{2})-\frac{16dr^{2}%
}{(k-1)^{2}} \\
&\geq &\alpha (x)\mu ^{k}-\beta (x)+\lambda (\rho ^{2}-r^{2})^{-\frac{2k}{k-1%
}}(\lambda ^{k-1}C(x)-c^{\prime }\rho ^{2}),
\end{eqnarray*}%
where $c^{\prime }=\Theta (\frac{4}{k-1}(2N+\frac{k+1}{k-1})+\frac{16d}{%
(k-1)^{2}}=c^{\prime }(N,p,q,k).$ We deduce that $\mathcal{L}(w)\geq 0$ if we impose%
\begin{equation*}
\mu ^{k}\geq \max_{B_\rho (x_{0})}\frac{\beta }{\alpha }\;\text{ and }\;\lambda
^{k-1}\geq c^{\prime }\rho ^{2}\max_{B_\rho (x_{0})}\frac{1}{\alpha }.
\end{equation*}%
If $x_{1}\in G_{1}$ is such that $z(x_{1})-w(x_{1})=\max_{G_{1}}(z-w)>0,$
then $\nabla z(x_{1})=\nabla w(x_{1}),$ and $\mathcal{A}(z-w)(x_{1})\geq 0.$
Therefore
\begin{equation*}
0\geq \mathcal{L}(z-w)(x_{1}))=\mathcal{A}(z-w)(x_{1})+\alpha
(x)(z^{k}-w^{k})(x_{1})+d\left(\frac{\left\vert \nabla w\right\vert ^{2}}{w}-%
\frac{\left\vert \nabla z\right\vert ^{2}}{z}\right).
\end{equation*}%
Since the last term is positive, it is a contradiction. Then $z\leq w$ in $%
B_\rho (x_{0}).$ In particular $z(x_{0})\leq w(x_{0}).$
\end{proof}\medskip

We also use a bootstrap argument, initialy used in \cite[Lemma 2.2]{BiGr1}, and
then in \cite{BiGaYa} in more general form. 

\begin{lemma}
\label{boot} Let $d,h\in \mathbb{R}$ with $d\in \left( 0,1\right) $ and $y$ be
a positive nondecreasing function on some interval $(r_{1},\infty )$. 
Assume that there exist $K>0$ and $\varepsilon _{0}>0$ such
that, for any $\varepsilon \in \left( 0,\varepsilon _{0}\right] $ and $%
r>r_{1}$,%
\begin{equation*}
y(r)\leq K\varepsilon ^{-h}y^{d}(r(1+\varepsilon )).
\end{equation*}%
Then there exists $C=C(K,d,h,\varepsilon _{0})$ such that $%
\sup_{(r_{1},\infty )}y\leq C.$
\end{lemma}

\noindent\begin{proof}
Consider the sequence $\{\varepsilon _{n}\}=\{\varepsilon _{0}2^{-n}\}_{n\geq 1}$. Since the series $\sum \varepsilon _{n}$ is
convergent,   the sequence
$\{P_m\}:=\{\prod_{i=1}^{m}(1+\varepsilon _{i})\}_{m\geq 1}$ is
convergent too, with limit $P>0$. Then there holds for any $r>r_{1}$ 
\begin{equation*}
y(r)\leq K\varepsilon _{1}^{-h}y^{d}(r(1+\varepsilon _{1}))=K\varepsilon
_{1}^{-h}y^{d}(rP_{1}).
\end{equation*}%
We deduce by induction, 
\begin{equation*}
y(r)\leq K^{1+d+..+d^{m}}(\varepsilon _{1}^{-h}\varepsilon
_{2}^{-hd}..\varepsilon _{m}^{-hd^{m-1}})y^{d^{m}}(P_{m}r)=(K\varepsilon
_{0}^{-h})^{1+d+..+d^{m}}(2^{h(1+2d+..+md^{m-1})})y^{d^{m}}(P_{m}r),
\end{equation*}%
and $rP_{m}\rightarrow rP ,$ $d^{m}\rightarrow 0,$ thus $%
\left(y(P_{m}r)\right)^{d^{m}}\rightarrow 1.$ Therefore we deduce that for any $r>r_{1}$, 
\begin{equation*}
y(r)\leq (K\varepsilon _{0}^{-h})^{\sum_{m=1}^{\infty
}d^{m}}2^{\sum_{m=1}^{\infty }md^{m-1}}=(K\varepsilon _{0}^{-h})^{\frac{1}{1-d}}2^{\frac{d}{(1-d)^2}}.
\end{equation*}
\end{proof}\smallskip

We also mention below a property of $m$-subharmonic functions given in \cite[Lemma 2.1]{BiGaYa}.
It's proof is also based  upon a boostrap method and is valid for more general
quasilinear operators:

\begin{lemma} 
\label{subha}Let $u\in W_{loc}^{1,m}(\Omega )$ be nonnegative, $m$-subharmonic function  
in a domain $\Omega $ of $\mathbb{R}^{N}.$ Then for any $\tau >0,$
there exists a constant $C=C(N,m,\tau )>0$ such that for any ball $%
B_{2\rho}(x_{0})\subset \Omega $ and any $\varepsilon \in \left( 0,\frac{1}{2}%
\right] ,$%
\begin{equation*}
\sup_{B_{\rho}(x_{0})}u\leq C\varepsilon ^{-\frac{Nm^{2}}{\tau ^{2}}}\left(\frac{1}{%
\left\vert B_{(1+\varepsilon)\rho}(x_{0})\right\vert }\int_{B_{(1+\varepsilon)\rho}(x_{0})}u^{\tau }\right)^{\frac{1}{\tau }}.
\end{equation*}

\end{lemma}

Finally we use some simple properties of mean value on spheres of $m$-superharmonic functions, in the same spirit as the ones given in \cite[Lemmas
3.7, 3.8, 3.9]{AmSi} for mean values on annulus, and in \cite{BuGMQu} for $m=2.$
For the sake of completeness we recall their proofs.

\begin{lemma}
\label{mude}Let $u\in C^1(\Omega)$ be nonnegative, $m$-superharmonic in $\Omega .$\smallskip

\noindent (i) If $\Omega =$ $\mathbb{R}^{N}\backslash B_{r_{0}}$, then 
\begin{equation*}
r\mapsto \mu (r): =\inf_{\left\vert x\right\vert =r}u,
\end{equation*}
is bounded in $(r_{0},\infty ),$ and strictly monotone or constant for large 
$r$. \smallskip

\noindent (ii) If $\Omega =B_{r_{0}}\backslash \left\{ 0\right\} ,$ then $r\mapsto\mu (r)$ is
nonincreasing in $(0,r_{0})$.
\end{lemma}

\noindent\begin{proof}
(i) Let $r>r_{0}$ be fixed. The function $f(x)=\mu (r)(1-(\frac{\left\vert
x\right\vert }{r_{0}})^{\frac{m-N}{m-1}})$ is $m$-harmonic, and $u\geq f$ on 
$\partial B_{r}\cup \partial B_{r_{0}}$ , therefore $u\geq f$ in $\overline{%
B_{r}}\backslash B_{r_{0}}.$ Let $k>0$ large enough such that $1-k^{\frac{%
p-N}{p-1}}\geq \frac{1}{2}$. If we take $r>kr_{0}$ and any $x$ such that $%
\left\vert x\right\vert =kr_{0}$ we obtain 
\begin{equation*}
u(x)\geq \mu (kr_{0})\geq f(x)=\mu (r)(1-k^{\frac{p-N}{p-1}})\geq \frac{1}{2}\mu (r),
\end{equation*}%
so $\mu (r)$ is bounded for $r>kr_{0}.$ For any $r_{2}>r_{1}>r_{0},$ $%
\varphi (r_{1},r_{2}):=\inf_{\overline{B_{r_{2}}}\backslash B_{1}}u=\min
(\mu (r_{1}),\mu (r_{2}))$ from the maximum principle. Then $\varphi $
is nonincreasing in $r_{2}$ and nondecreasing in $r_{1}.$ If $\mu $ has a
strict local minimum at some point $r,$ then for $0<\delta <\delta _{0}$
small enough, $\mu (r)<$ $\varphi (r-\delta _{0},r+\delta _{0})\leq \varphi
(r-\delta ,r+\delta ),$ which yields a contradiction as $\delta \rightarrow 0$. Then $\mu $ is monotone.
If it is constant on two intervals $(a,b)$ and $(a',b')$ with $b<a'$ and non-constant on $(b,a')$ it follows by Vazquez's maximum principle \cite{Vaz} that
$u$ is constant on $\overline B_b\setminus  B_a$ and on  $\overline B_{b'}\setminus  B_{a'}$ but non constant on $B_{a'}\setminus \overline B_b$. It means, always by Vazquez's maximum principle, \smallskip

\noindent - either $\min \{\mu (r):a<r<b'\}=\mu (a)$ (if $\mu$ is nondecreasing) and the minimum of $u$ in $\overline B_{b'}\setminus  B_a$ is achieved 
in any point in $B_b\setminus  \overline B_a$, hence $u$ is constant in $\overline B_{b'}\setminus  B_a$,\smallskip

\noindent - or $\min \{\mu (r):a<r<b'\}=\mu (a)$  (if $\mu$ is nonincreasing) and the minimum of $u$ in $\overline B_{b'}\setminus  B_a$ is achieved 
in any point in $B_{b'}\setminus  \overline B_{a'}$, hence $u$ is constant in $\overline B_{b'}\setminus  B_a$.\smallskip

\noindent In both case we obtain a contradiction. Hence $\mu$ is either strictly monotone for $r$ large enough, or it is constant, and so is $u$.\smallskip
 
(ii) For given $r_{1}<r_{0},$ and $\delta >0,$ there exists $\varepsilon
_{\delta }\leq r_1$ such that for $0<\varepsilon <\varepsilon _{\delta }$, such that 
$\delta \varepsilon ^{\frac{m-N}{m-1}}\geq \mu (r_{1}).$ Let $h(x)=\mu
(r_{1})-\delta \left\vert x\right\vert ^{\frac{m-N}{m-1}}.$ Then $u\geq h$
on $\partial B_{r_{1}}\cup \partial B_{\varepsilon },$ then $u\geq h$ in $%
\overline{B_{r_{1}}}\backslash B_{\varepsilon }$. Making $\varepsilon
\rightarrow 0$ and then $\delta \rightarrow 0,$ one gets $u\geq \mu (r_{1})$
in $B_{r_{1}}\backslash \left\{ 0\right\} ,$ thus $\mu (r)\geq \mu (r_{1})$
for $r<r_{1}.$
\end{proof}
\section{Proof of the main results}
\subsection{Proof of the Liouville property for $q>m$\label{Liou}}

We first give a general Bernstein estimate for solutions of equation (\ref%
{eqm}):

\begin{lemma}
\label{Bern} Let $u$ be any $C^{1}$ positive solution of ((\ref{eqm}) in a
domain $\Omega $, with $m>1$ and $p,q$   arbitrary real numbers. Let $%
G=\left\{ x\in \Omega :\left\vert \nabla u(x)\right\vert \neq 0\right\} .$
Let $u=v^{b}$ with $b\in \mathbb{R}\backslash \left\{ 0\right\} $ and $%
z=\left\vert \nabla v\right\vert ^{2}.$ Then the operator 
\begin{equation}
w\mapsto \mathcal{A}(w)=-\Delta w-(m-2)\frac{D^{2}w(\nabla v,\nabla v)}{%
\left\vert \nabla v\right\vert ^{2}}=-\sum_{i,j=1}^{N}a_{ij}v_{x_{i}x_{j}},  \label{ope}
\end{equation}%
with coefficients $a_{i,j}$ depending on $\nabla v,$ is uniformly elliptic
in $G,$ and for any $\varepsilon >0,$ there exists $C_{\varepsilon
}=C_{\varepsilon }(N,m,p,q,b,\varepsilon )$ such that 
\begin{equation}\label{oko}\begin{array}{lll}\displaystyle
-\frac{1}{2}\mathcal{A}(z)+\left(\frac{1-\varepsilon }{N}%
(b-1)^{2}(m-1)^{2}-(1-b)(m-1)\right)\frac{z^{2}}{v^{2}}
+\frac{1-2\varepsilon }{N}\left\vert b\right\vert
^{2(q-m+1)}v^{2s}z^{q+2-m}\\[4mm]
\displaystyle\phantom{----}
+\left(\frac{1}{N}2(b-1)(m-1)-s\right)\left\vert b\right\vert
^{q-m}bv^{s-1}z^{\frac{q+4-m}{2}}\leq C_{\varepsilon }\frac{\left\vert
\nabla z\right\vert ^{2}}{z}.  
\end{array}\end{equation}
\end{lemma}

\noindent\begin{proof}
The following identities hold if $u=v^b$: $\nabla u=bv^{b-1}\nabla v,$%
\begin{equation*}
\left\vert \nabla u\right\vert ^{m-2}\nabla u=\left\vert b\right\vert
^{m-2}bv^{(b-1)(m-1)}\left\vert \nabla v\right\vert ^{m-2}\nabla v,
\end{equation*}%
\begin{equation*}
\Delta _{m}u=\left\vert b\right\vert ^{m-2}b(v^{(b-1)(m-1)}\Delta
_{m}v+(b-1)(m-1)v^{(b(m-1)-m}\left\vert \nabla v\right\vert ^{m}),
\end{equation*}%

\begin{equation*}
-v^{(b-1)(m-1)}\Delta _{m}v=(b-1)(m-1)v^{(b(m-1)-m}\left\vert \nabla
v\right\vert ^{m}+\left\vert b\right\vert ^{q}v^{bp+(b-1)q}\left\vert \nabla
v\right\vert ^{q},
\end{equation*}%
and finally
\begin{equation}
-\Delta _{m}v=(b-1)(m-1)\frac{\left\vert \nabla v\right\vert ^{m}}{v}%
+\left\vert b\right\vert ^{q-m}bv^{s}\left\vert \nabla v\right\vert ^{q},
\label{eqv}
\end{equation}%
with 
\begin{equation}
s=m-1-q+b(p+q-m+1).  \label{cos}
\end{equation}%
We set $z=\left\vert \nabla v\right\vert ^{2}.$ Then in $G,$ 
\begin{equation*}
-\Delta _{m}v=f\Longleftrightarrow -\Delta v-(m-2)\frac{D^{2}v(\nabla
v,\nabla v)}{\left\vert \nabla v\right\vert ^{2}}=f\left\vert \nabla
v\right\vert ^{2-m},
\end{equation*}%
from which identity we infer 
\begin{equation*}
-\Delta v=(m-2)\frac{D^{2}v(\nabla v,\nabla v)}{\left\vert \nabla
v\right\vert ^{2}}+(b-1)(m-1)\frac{\left\vert \nabla v\right\vert ^{2}}{v}%
+\left\vert b\right\vert ^{q-m}bv^{s}\left\vert \nabla v\right\vert ^{q+2-m}
\end{equation*}%
where 
\begin{equation*}
<Hess\, v(\nabla v),\nabla v>=D^{2}v(\nabla v,\nabla v)=\frac{1}{2}<\nabla z,\nabla v>.
\end{equation*}%
We recall the B\"{o}chner formula combined with Cauchy-Schwarz inequality, 
\begin{equation*}
-\frac{1}{2}\Delta z+\frac{1}{N}(\Delta v)^{2}+<\nabla (\Delta v),\nabla
v)\leq -\frac{1}{2}\Delta z+(Hess\,v)^{2}+<\nabla (\Delta v),\nabla v>=0.
\end{equation*}%
Since
\begin{equation*}
-\Delta v=\frac{m-2}{2}\frac{<\nabla z,\nabla v>}{z}+(b-1)(m-1)\frac{z}{v}%
+\left\vert b\right\vert ^{q-m}bv^{s}z^{\frac{q+2-m}{2}},
\end{equation*}%
we deduce
\begin{equation*}\begin{array}{lll}\displaystyle
<\nabla (\Delta v),\nabla v>=-\frac{m-2}{2}<\nabla \frac{<\nabla z,\nabla
v>}{z},\nabla v> +(1-b)(m-1) <\nabla \frac{z}{v},\nabla v>\\[3mm]
\phantom{<\nabla (\Delta v),\nabla v>}\displaystyle
-\left\vert b\right\vert
^{q-m}b(sv^{s-1}z^{\frac{q+4-m}{2}}+\frac{q+2-m}{2}v^{s}z^{\frac{q-m}{2}%
}<\nabla z,\nabla v>).
\end{array}\end{equation*}%
we observe that 
\begin{equation*}
<\nabla \frac{z}{v},\nabla v>=\frac{<\nabla z,\nabla v>}{v}-\frac{z^{2}}{%
v^{2}}\quad \text{and }\;\frac{<\nabla z,\nabla v>^{2}}{z^{2}}\leq \frac{\left\vert
\nabla z\right\vert ^{2}}{z},
\end{equation*}%
thus 
\begin{equation*}\begin{array}{lll}\displaystyle
-\frac{m-2}{2} <\nabla \frac{<\nabla z,\nabla v>}{z},\nabla v>=-\frac{m-2}{
2}\left(\frac{D^{2}z(\nabla v,\nabla v)}{z}+\frac{1}{2}\frac{\left\vert \nabla
z\right\vert ^{2}}{z}-\frac{<\nabla z,\nabla v>^{2}}{z^{2}}\right) \\[4mm]
\phantom{\displaystyle
-\frac{m-2}{2} <\nabla \frac{<\nabla z,\nabla v>}{z},\nabla v>}
\displaystyle
\geq -\frac{m-2}{2}\frac{D^{2}z(\nabla v,\nabla v)}{z}-\left\vert
m-2\right\vert \frac{\left\vert \nabla z\right\vert ^{2}}{z}.
\end{array}\end{equation*}%
We define the operator $\mathcal{A}$ by (\ref{ope}); it satisfies (\ref{kli}%
) with $\theta =\min (1,m-1)$ and $\Theta =\max (1,m-1),$ so it is uniformly
elliptic in $G.$ Therefore 
\begin{equation}\label{estiv}\begin{array}{lll}\displaystyle
-\frac{1}{2}\mathcal{A}(z)+\frac{1}{N}(\Delta v)^{2}-(1-b)(m-1)\frac{z^{2}%
}{v^{2}}-\left\vert b\right\vert ^{q-m}bsv^{s-1}z^{\frac{q+4-m}{2}}  
\\[3mm]\phantom{----}
\displaystyle
\leq (b-1)(m-1)\frac{<\nabla z,\nabla v>}{v}+\frac{(q+2-m)\left\vert
b\right\vert ^{q-m}b}{2}v^{s}z^{\frac{q-m}{2}}<\nabla z,\nabla v>\\[3mm]\phantom{----}
\displaystyle+\left\vert
m-2\right\vert \frac{\left\vert \nabla z\right\vert ^{2}}{z}. 
\end{array}\end{equation}%
For $\varepsilon >0$ there holds by H\"older's inequality, 
\begin{equation*}
\frac{q+2-m}{2}v^{s}z^{\frac{q-m}{2}}<\nabla z,\nabla v>\leq \frac{%
\varepsilon }{N}\left\vert b\right\vert
^{2(q-m+1)}v^{2s}z^{q+2-m}+C_{\varepsilon }\frac{\left\vert \nabla
z\right\vert ^{2}}{z},
\end{equation*}
\begin{eqnarray*}
(\Delta v)^{2} &=&\left(\frac{m-2}{2}\frac{<\nabla z,\nabla v>}{z}+(b-1)(m-1)%
\frac{z}{v}+\left\vert b\right\vert ^{q-m}bv^{s}z^{\frac{q+2-m}{2}}\right)^{2} \\
&\geq &(b-1)^{2}(m-1)^{2}\frac{z^{2}}{v^{2}}+\left\vert b\right\vert
^{2(q-m+1)}v^{2s}z^{q+2-m}+2(b-1)(m-1)\left\vert b\right\vert
^{q-m}bv^{s-1}z^{\frac{q+4-m}{2}} \\
&&-(m-2)\frac{\left\vert \nabla z\right\vert }{\sqrt{z}}(\left\vert
b-1\right\vert (m-1)\frac{z}{v}+\left\vert b\right\vert ^{q-m+1}v^{s}z^{%
\frac{q+2-m}{2}}),
\end{eqnarray*}%
and for any $\varepsilon >0,$%
\begin{equation*}
(m-2)\frac{\left\vert \nabla z\right\vert }{\sqrt{z}}(\left\vert
b-1\right\vert (m-1)\frac{z}{v}\leq \frac{\varepsilon }{N}(b-1)^{2}(m-1)^{2}%
\frac{z^{2}}{v^{2}}+C_{\varepsilon }\frac{\left\vert \nabla z\right\vert ^{2},%
}{z}
\end{equation*}%
\begin{equation*}
(m-2)\frac{\left\vert \nabla z\right\vert }{\sqrt{z}}\left\vert b\right\vert
^{q-m+1}v^{s}z^{\frac{q+2-m}{2}}\leq \frac{\varepsilon }{N}\left\vert
b\right\vert ^{2(q-m+1)}v^{2s}z^{q+2-m}+C_{\varepsilon }\frac{\left\vert
\nabla z\right\vert ^{2}}{z},
\end{equation*}%
thus (\ref{oko}) follows.\medskip
\end{proof}

\noindent\begin{proof}[Proof of Theorem \protect\ref{main}]
We use Lemma \ref{Bern} with $b\in (0,1),$ combined with the estimate 
\begin{equation*}
\left(\frac{1}{N}2(b-1)(m-1)-s\right)\left\vert b\right\vert ^{q-m}bv^{s-1}z^{\frac{%
q+4-m}{2}}\leq \frac{\varepsilon }{N}\left\vert b\right\vert
^{2(q-m+1)}v^{2s}z^{q+2-m}+C_{\varepsilon }\frac{z^{2}}{v^{2}}.
\end{equation*}%
Then there exist constants $C_{i}>0$ depending only on $m,b,N,p,q,$ such that
\begin{equation*}
\frac{1}{2}\mathcal{A}(z)+C_{1}v^{2s}z^{q+2-m}\leq C_{2}\frac{z^{2}}{v^{2}}%
+C_{3}\frac{\left\vert \nabla z\right\vert ^{2}}{z}.
\end{equation*}%
Next we choose $s=-1$ in (\ref{cos}), thus 
\begin{equation*}
b(p+q-m+1)=q-m,
\end{equation*}%
which is positive because $q>m.$ We deduce using H\"{o}lder inequality, 
\begin{equation*}
\mathcal{A}(z)+\frac{C_{4}z^{q+2-m}-C_{5}}{v^{2}}\leq \mathcal{A}(z)+\frac{%
C_{1}z^{q+2-m}-C_{2}z^{2}}{v^{2}}\leq C_{3}\frac{\left\vert \nabla
z\right\vert ^{2}}{z}.
\end{equation*}%
If we apply Lemma \ref{oss} with 
\begin{equation*}
\alpha (x)=\frac{C_{4}}{v^{2}(x)},\qquad \beta (x)=\frac{C_{5}}{v^{2}(x)}%
,\qquad k=q+2-m,
\end{equation*}%
we deduce that any solution in $\overline{B}_\rho(x_{0})$ $\rho >0,$
satisfies 
\begin{equation*}
z(x_{0})\leq C_{6}\left(\frac{1}{\alpha \rho ^{2}}\right)^{\frac{1}{k-1}}+\left(\frac{C_{5}}{%
C_{4}}\right)^{\frac{1}{k}}\leq C_{7}\left(\frac{\max_{B_\rho (x_{0})}v}{\rho }\right)^{\frac{%
2}{q+1-m}}+1),
\end{equation*}%
which yields
\begin{equation}
\left\vert \nabla v(x_{0})\right\vert \leq C_{8}\left(\left(\frac{\max_{\overline{B}_\rho(x_{0})}v}{\rho }\right)^{\frac{1}{q+1-m}}+1\right),  \label{maxa}
\end{equation}%
where we observe that $\frac{1}{q+1-m}<1,$ since $q>m.$ Let $\varepsilon \in
\left( 0,\frac{1}{2}\right] .$ As a consequence, for any solution in $B_{2R}$%
, (or even $\overline{B}_{\frac{3R}{2}})$ considering any $x_{0}\in 
\overline{B}_{R}$ and taking $\rho =R\varepsilon ,$ we get 
\begin{equation}
\max_{\overline{B}_{R}}\left\vert \nabla v\right\vert \leq c\left(\left(\frac{\max_{%
\bar{B}_{R(1+\varepsilon )}}v}{\varepsilon R}\right)^{\frac{1}{q+1-m}}\!\!+1\right)\leq
c\varepsilon ^{-\frac{1}{q+1-m}}\left(\left(\frac{\max_{\bar{B}_{R(1+\varepsilon )}}v}{%
R}\right)^{\frac{1}{q+1-m}}\!\!+1\right),  \label{maxo}
\end{equation}%
\begin{equation*}
\max_{\bar{B}_{R(1+\varepsilon )}}v\leq v(0)+R(1+\varepsilon )\max_{\bar{B}%
_{R(1+\varepsilon )}}\left\vert \nabla v\right\vert,
\end{equation*}%
\begin{equation*}
\frac{\max_{\bar{B}_{R(1+\varepsilon )}}v}{R}\leq \frac{1+v(0)}{R}%
+(1+\varepsilon )\max_{\bar{B}_{R(1+\varepsilon )}}\left\vert \nabla
v\right\vert \leq c_{0}\left(\frac{1}{R}+\max_{\bar{B}_{R(1+\varepsilon
)}}\left\vert \nabla v\right\vert \right),
\end{equation*}%
where $c_{0}=2+v(0)$ depends on $v(0).$ If $R\geq 1,$ 
\begin{equation*}
\left(\frac{\max_{\bar{B}_{R(1+\varepsilon )}}v}{R}\right)^{\frac{1}{q+1-m}}+1\leq
c_{0}^{\frac{1}{q+1-m}}\left(1+\max_{\bar{B}_{R(1+\varepsilon )}}\left\vert
\nabla v\right\vert )^{\frac{1}{q+1-m}}\right)+1\leq c_{1}\left(1+\max_{\bar{B}%
_{R(1+\varepsilon )}}\left\vert \nabla v\right\vert \right)^{\frac{1}{q+1-m}}.
\end{equation*}%
Then from (\ref{maxo}), 
\begin{equation*}
y(R):=1+\max_{\overline{B}_{R}}\left\vert \nabla v\right\vert \leq
1+c_{2}\varepsilon ^{-\frac{1}{q+1-m}}(1+\max_{\bar{B}_{R(1+\varepsilon
)}}\left\vert \nabla v\right\vert )^{\frac{1}{q+1-m}}\leq c_{3}\varepsilon
^{-\frac{1}{q+1-m}}(1+\max_{\bar{B}_{R(1+\varepsilon )}}\left\vert \nabla
v\right\vert )^{\frac{1}{q+1-m}}.
\end{equation*}%
Using the definition of $y$, this is 
\begin{equation*}
y(R)\leq c\varepsilon ^{-\frac{1}{q+1-m}}\left(y((1+\varepsilon
)R)\right)^{\frac{1}{q+1-m}},
\end{equation*}%
where $c$ depends on $v(0).$ Therefore $y(R)$ is bounded as a consequence of Lemma \ref{boot}.
Thus $\left\vert \nabla v\right\vert $ is bounded and using the definition of $v$ with the value of $b$,  
\begin{equation*}
\left\vert \nabla v\right\vert ^{q-m}=u^{p+1}\left\vert \nabla u\right\vert
^{q-m}\in L^{\infty }(\mathbb{R}^{N}).
\end{equation*}%
Next we consider any $l\geq 0$ such that $u-l>0.$ The function $u_{l}=u-l$
satisfies 
\begin{equation*}
0\leq -\Delta _{m}u_{l}\leq C_{\infty }\frac{\left\vert \nabla u\right\vert
^{m}}{u}\leq C_{\infty }\frac{\left\vert \nabla u_{l}\right\vert ^{m}}{u_{l}},
\end{equation*}%
with $C_{\infty }=\left\Vert u^{p+1}\left\vert \nabla u\right\vert
^{q-m}\right\Vert _{L^{\infty }(\mathbb{R}^{N})}$.  Then the function  $w_{l}=u_{l}^{\sigma }$ with $\sigma >1$ to be specified below, satisfies 
\begin{eqnarray*}
-\Delta _{m}w_{l} &=&\sigma ^{m-1}u_{l}{}^{(\sigma -1)(m-1)}(-\Delta
_{m}u_{l}+(\sigma -1)(m-1)\frac{\left\vert \nabla u_{l}\right\vert ^{m}}{%
u_{l}}) \\
&\leq &\sigma ^{m-1}((\sigma -1)(m-1)-C_{\infty })u_{l}^{\sigma
(m-1)-m}\left\vert \nabla u_{l}\right\vert ^{m}.
\end{eqnarray*}%
Therefore $w_{l}$ is $m$-subharmonic for $\sigma $ large enough.

We first take $l=0,$ so $w=u^{\sigma }.$ By Lemma \ref{subha}, for any $\tau
>0,$ there exists a constant $C_{\tau }=C_{\tau }(N,m,\tau )$ such that 
\begin{equation}\label{lau1}
\sup_{B_{R}}w\leq C_{\tau }\left(\frac{1}{|B_{2R}|}\int_{B_{2R}}w^{\tau }\right)^{\frac{%
1}{\tau }}=C_{\tau }\left(\frac{1}{|B_{2R}|}\int_{B_{2R}}u^{\tau \sigma }\right)^{\frac{%
1}{\tau }},
\end{equation}%
and since $u$ is $m$-superharmonic, there holds for any $\theta \in (0,\frac{N(m-1)}{N-m%
}),$ \cite{Tru}
\begin{equation}\label{lau2}
\inf_{B_{R}}u\geq c_{\theta }\left(\frac{1}{|B_{2R}|}\int_{B_{2R}}u^{\theta }\right)^{%
\frac{1}{\theta }}.
\end{equation}%
Taking $\tau =\frac{\theta }{\sigma },$ we deduce 
\begin{equation}
\sup_{B_{R}}u=(\sup_{B_{R}}w)^{\frac{1}{\sigma }}\leq C_{\tau }^{\frac{1}{%
\sigma }}\left(\frac{1}{|B_{2R}|}\int_{B_{2R}}u^{\tau \sigma }\right)^{\frac{1}{s\sigma 
}}\leq \frac{C_{\tau }^{\frac{1}{\sigma }}}{c_{\theta }}\inf_{B_{R}}u.
\label{des}
\end{equation}%
This means that $u$, and also $%
w$,  satisfies the Harnack inequality in $\mathbb{R}^{N}$: 
\begin{equation*}
\sup_{B_{R}}w\leq \frac{C_{\tau }}{c_{\theta }^{\sigma }}\inf_{B_{R}}w.
\end{equation*}%
But $r\mapsto \mu (r)=\inf_{\left\vert x\right\vert =r}u=\inf_{B_{r}}u$ from
the maximum principle, is nonincreasing, so it has a limit $L\geq 0$ as $%
r\rightarrow \infty $. This implies that $u$ is bounded and $l=\inf_{%
\mathbb{R}^{N}}u\geq 0$. If we replace $u$ by $u_{l}$ and $w$ by $w_{l}$, 
then  (\ref{lau1}) holds with $w$ and $u$ replaced respectively by $w_l$ and $u_l$ since $w_{l}$ is m-subharmonic, but 
also  (\ref{lau2}) holds with $u$ replaced by $u_l$ since $u_l$ is m-superharmonic. Thus
\begin{equation*}
\sup_{B_{R}}w_{l}\leq C(\inf_{B_{R}}u_{l})^{\sigma }.
\end{equation*}%
Therefore $\sup_{B_{R}}w_{l}$ tends to $0$ as $R\rightarrow \infty .$ Then $%
w_{l}\equiv 0,$ thus $u\equiv l.$
\end{proof}

\subsection{Asymptotic behaviour near $\infty $}

In this section we consider the behaviour of solutions defined in an exterior domain.\medskip

\noindent\begin{proof}[Proof of Theorem \protect\ref{exte}]
Consider a nonnegative solution $u=v^b$ ($0<b<1$) of (\ref{eqm}) in $\mathbb{R}^{N}\backslash B_{r_{0}}$. From (\ref{maxa}%
) the function $v$ satisfies  in $\overline{B}_\rho (x_{0})$ ($\rho >0$),
\begin{equation}
\left\vert \nabla v(x_{0})\right\vert \leq C\left(\left(\frac{\max_{\overline{B}_\rho(x_{0})}v}{\rho }\right)^{\frac{1}{q+1-m}}+1\right),  \label{app}
\end{equation}%
with $C=C(N,p,q,m).$ Here we denote by $c_{i}$ some positive constants depending on $%
r_{0},N,p,q,m.$ Let $R>4r_{0}$ and $0<\varepsilon \leq \frac{1}{4}.$ Applying
 (\ref{app}) with $\rho =\varepsilon R,$ we get
\begin{equation*}
\left\vert \nabla v(x_{0})\right\vert\leq c_{1}\left(\left(\frac{\max_{\overline{B}_{\varepsilon R}(x_{0})}v}{\varepsilon R}\right)^{\frac{1}{q+1-m}}+1\right)\leq
c_{1}\varepsilon ^{-\frac{1}{q+1-m}}\left(\left(\frac{\max_{\overline{B}_{\varepsilon  R}%
(x_{0})}v}{R}\right)^{\frac{1}{q+1-m}}+1\right),
\end{equation*}%
then%
\begin{equation*}
\max_{\left\vert x\right\vert =R}\left\vert \nabla v\right\vert \leq c_{2}\left(\left(
\frac{\max_{R(1-\frac{\varepsilon }{2})\leq \left\vert x\right\vert \leq R(1+%
\frac{\varepsilon }{2})}v}{\varepsilon R}\right)^{\frac{1}{q+1-m}}+1\right)\leq c_{3}\left(\left(
\frac{\max_{\frac{R}{1+\varepsilon }\leq \left\vert x\right\vert \leq
R(1+\varepsilon )}v}{\varepsilon R}\right)^{\frac{1}{q+1-m}}+1\right),
\end{equation*}
\begin{equation*}
\max_{\frac{R}{2}\leq \left\vert x\right\vert \leq 2R}\left\vert \nabla
v\right\vert \leq c_{4}\left(\left(\frac{\max_{\frac{R}{2(1+\varepsilon )}\leq
\left\vert x\right\vert \leq 2R(1+\varepsilon )}v}{\varepsilon R}\right)^{\frac{1}{%
q+1-m}}+1\right),
\end{equation*}
and finally, 
\begin{equation*}
1+\max_{\frac{R}{2}\leq \left\vert x\right\vert \leq 2R}\left\vert \nabla
v\right\vert \leq c_{5}\varepsilon ^{-\frac{1}{q+1-m}}\left(\left(\frac{\max_{\frac{R}{%
2(1+\varepsilon )}\leq \left\vert x\right\vert \leq 2R(1+\varepsilon )}v}{R}%
\right)^{\frac{1}{q+1-m}}+1\right).
\end{equation*}%
From Lemma \ref{mude}-(i), $\mu (r)=\inf_{\left\vert
x\right\vert=r }u=(\inf_{\left\vert x\right\vert =r}v)^{b}$ is bounded : let 
$M=\max_{r\geq r_{0}}\mu (r).$ Note that $M$ depends on $u$. Now for any $x$
such that $\left\vert x\right\vert =\rho ,$ there exists at least one point $%
x_{\rho }$ where $v(x_{\rho })=\inf_{\left\vert x\right\vert =\rho }v$ . We
can join any point $x\in S_{\rho }$ to $x_{\rho }$ by a connected chain of balls of radius $\varepsilon\rho$ with at points $x_{i}\in S_{\rho }$  and this chain can be constructed so that it has at most $\frac{\pi}{\varepsilon}$ elements.
Considering one ball containing $x$ and joining it to a ball containing $%
x_{\rho },$ we get that 
\begin{equation*}
v(x)\leq v(x_{\rho })+C_{N}\varepsilon ^{-1}\rho \max_{\frac{\rho }{
1+\varepsilon }\leq \left\vert x\right\vert \leq \rho (1+\varepsilon
)}\left\vert \nabla v\right\vert \leq M^{\frac{1}{b}}+C_{N}\varepsilon
^{-1}\rho \max_{\frac{\rho }{
1+\varepsilon }\leq \left\vert x\right\vert \leq \rho (1+\varepsilon
)}\left\vert \nabla v\right\vert.
\end{equation*}%
Then 
\begin{equation*}\begin{array}{lll}\displaystyle
\max_{\frac{R}{2(1+\varepsilon )}\leq \left\vert x\right\vert \leq
2R(1+\varepsilon )}v\leq c_{M}^{1}\left(1+\varepsilon ^{-1}R\max_{\frac{R}{%
2(1+3\varepsilon )}\leq \left\vert x\right\vert \leq 2R(1+3\varepsilon
)}\left\vert \nabla v\right\vert \right)
\\[4mm]\phantom{\displaystyle\max_{\frac{R}{2(1+\varepsilon )}\leq \left\vert x\right\vert \leq
2R(1+\varepsilon )}v}
\displaystyle\leq c_{M}^{2}\varepsilon ^{-1}R\left(1+\max_{%
\frac{R}{2(1+3\varepsilon )}\leq \left\vert x\right\vert \leq
2R(1+3\varepsilon )}\left\vert \nabla v\right\vert \right),
\end{array}\end{equation*}
and
\begin{equation*}
\frac{1}{\varepsilon R}\max_{\frac{R}{2(1+\varepsilon )}\leq \left\vert x\right\vert \leq
2R(1+\varepsilon )}v\leq c_{M}^{3}\varepsilon
^{-2}\left(1+\max_{\frac{R}{2(1+3\varepsilon )}\leq \left\vert x\right\vert
\leq 2R(1+3\varepsilon )}\left\vert \nabla v\right\vert \right).
\end{equation*}
Using estimate (\ref{maxo}) we obtain
\begin{equation}\label{lau3}
1+\max_{\frac{R}{2}\leq \left\vert x\right\vert \leq 2R}\left\vert \nabla
v\right\vert \leq c_{M}^{4}\varepsilon ^{-\frac{2}{q+1-m}}\left(1+\max_{\frac{R%
}{2(1+3\varepsilon )}\leq \left\vert x\right\vert \leq 2R(1+3\varepsilon
)}\left\vert \nabla v\right\vert \right)^{\frac{1}{q+1-m}}.
\end{equation}%
Let $\{\varepsilon_n\}_{n\geq 1}$ be a positive decreasing sequence such that $P_n:=\prod_{j=1}^n(1+\varepsilon_j)\to 2$ and 
$\Theta_n:=\prod_{j=1}^n\varepsilon_{j+1}^{d^j}\to\Theta>0$ when $n\to\infty$. It is easy to find such sequences such that $\varepsilon_j\sim 2^{-j}$.  
For $\frac R2\leq a<2R\leq b$ we set 
$$y(a,b)=1+\max_{a\leq \left\vert x\right\vert \leq b}\left\vert \nabla v\right\vert\quad.
$$
Then (\ref{lau3}) with $(a,b)=(\frac{R}{2}, 2R)$ and $\varepsilon_1=3\varepsilon$ asserts that
$$y(\tfrac{R}{2}, 2R)\leq c_5\varepsilon_1^{-h}\left(y(\tfrac{R}{2(1+\varepsilon_1)}, 2R(1+\varepsilon_1))\right)^d\quad\text{with }\;h=\frac{2}{q+1-m}\,\text{ and }\;d=\frac{1}{q+1-m}\in (0,1).
$$
Applying (\ref{lau3}) with $(a,b)=(\frac{R}{2P_n}, 2RP_n)$ we obtain
$$y(\tfrac{R}{2P_n}, 2RP_n)\leq c_5\varepsilon_{n+1}^{-h}\left(y(\tfrac{R}{2P_{n+1}}, 2RP_{n+1})\right)^d.
$$
By induction, we deduce
\begin{equation}\label{lau4}
y(\tfrac{R}{2}, 2R)\leq c_5^{1+d+d^2+...+d^n}\Theta_n^{-h}\left(y(\tfrac{R}{2P_{n+1}}, 2RP_{n+1})\right)^{d^{n+1}}.
\end{equation}
Since $y(\tfrac{R}{2P_{n+1}}, 2RP_{n+1})\to y(\tfrac{R}{4}, 4R)$, we obtain that 
\begin{equation}\label{lau5}
1+\max_{\frac R2\leq \left\vert x\right\vert \leq 2R}\left\vert \nabla v\right\vert \leq c_5^{\frac{d}{1-d}}\Theta^{-h}:=C(M,N,p,q,m).
\end{equation}

Then we conclude again that $\left\vert \nabla v\right\vert $ is bounded for 
$\left\vert x\right\vert \geq 4r_{0},$ then in $\mathbb{R}^{N}\backslash
B_{r_{0}}$ since we have  assumed that $u\in C^{1}\mathbb{R}^{N}\backslash
B_{r_{0}}.$ We consider again the function $w=u^{\sigma },$ for $\sigma $
depending of $r_{0},$ large enough so that $(\sigma -1)(m-1)\geq \left\Vert
u^{p+1}\left\vert \nabla u\right\vert ^{q-m}\right\Vert _{L^{\infty }(%
\mathbb{R}^{N}\backslash B_{r_{0}})}.$ As in the proof of Theorem %
\ref{main} we conclude that $w$ is $m$-subharmonic in $\mathbb{R}%
^{N}\backslash \bar{B}_{r_{0}}$. Hence $u$ satisfies the Harnack inequality, using 
the estimate (\ref{des}). Therefore, for any $R>2r_{0}$, 
\begin{equation*}
\sup_{\frac{R}{2}\leq \left\vert x\right\vert \leq 3\frac{R}{2}}u\leq C\inf_{%
\frac{R}{2}\leq \left\vert x\right\vert \leq 3\frac{R}{2}}u.
\end{equation*}%
Since $u$ is is $m$-superharmonic, it follows by the strong maximum principle, that it cannot have any local
minimum in $\mathbb{R}^{N}\backslash \bar{B}_{r_{0}}$. Since  $u^{\sigma }$ is m-subharmonic it cannot 
have any local maximum too, and $u$ shares this property. As a consequence $%
\left\vert \nabla u\right\vert $\ does not vanish in $\mathbb{R}%
^{N}\backslash \bar{B}_{r_{0}}$. The function $r\mapsto \mu $ is bounded by Lemma \ref%
{mude}, hence $u$ is also bounded by the above Harnack inequality. Finally $\mu (r)$ is
monotone for large $r,$ so it admits a limit $l\geq 0$ when $r\to\infty$.

If $\mu (r)$ is nonincreasing for $r\geq r_{1}>r_{0}$, then $u-l\geq 0,$ so
we can consider the function $w_{l}$ instead of $w.$ Then 
\begin{equation*}
\max_{R\leq \left\vert x\right\vert \leq 2R}w_{l}\leq C(\inf_{R\leq
\left\vert x\right\vert \leq 2R}(u-l))^{\sigma },
\end{equation*}%
Then $w_{l}$ tends to $0,$ thus $u$ tends to $l$ as $\left\vert x\right\vert
\rightarrow \infty $. Since $u-l$ is $m$-superharmonic in $\mathbb{R}%
^{N}\backslash B_{r_{0}},$ then there holds 
\begin{equation*}
u(x)-l\geq C\left\vert x\right\vert ^{\frac{m-N}{m-1}},
\end{equation*}%
with $C=C(r_{0},N,m,u),$ see for example \cite[Proposition 2.6]{BiPo}, \cite[
Lemma 2.3]{SeZo}. It is the case in particular when $u$\ is a solution in%
\textbf{\ }$\mathbb{R}^{N}\backslash \left\{ 0\right\}$. Note
that the radial solutions such that $\mu $ is nonincreasing are
precisely defined in $(0,\infty )$.

Now, it follows from the upper estimate of $y(R),$ that the function $u$ satisfies
\begin{equation*}
-\Delta _{m}u=u^{p}\left\vert \nabla u\right\vert ^{q}\leq C\frac{\left\vert
\nabla u\right\vert ^{m}}{u},
\end{equation*}
 in $\mathbb{R}^{N}\backslash B_{r_{0}}$. Next suppose that $l>0.$ Then 
\begin{equation*}
-\Delta _{m}u\leq C^{\prime }\left\vert \nabla u\right\vert ^{m}.
\end{equation*}%
The function $U$ (still used in case $q=m),$ defined by 
\begin{equation*}
U=(m-1)(e^{\frac{u-l}{m-1}}-1),
\end{equation*}%
satisfies $-\Delta _{m}U\leq 0$ and $U$ tend to $0$ at $\infty$. Then there
exists $R_{\varepsilon }>0$ such that $U(x)\leq \varepsilon $ for $%
\left\vert x\right\vert \geq R_{\varepsilon }.$ For $R>R_{\varepsilon },$ the function 
$x\mapsto\omega(x) :=\varepsilon
+(\sup_{\left\vert z\right\vert =r_{0}}U(z))(\frac{\left\vert x\right\vert }{%
r_{0}})^{\frac{m-N}{m-1}}$ is a $m$-harmonic in $B_{R}\backslash
B_{r_{0}} $, hence it is larger than  $U$. Letting $%
\varepsilon \rightarrow 0$ we get $U\leq C\left\vert x\right\vert ^{\frac{m-N%
}{m-1}}$ near $\infty ;$ and $U$ has the same behaviour as $u-l$ $,$ so we
deduce the estimate from above,
\begin{equation*}
u(x)-l\leq C\left\vert x\right\vert ^{\frac{m-N}{m-1}}.
\end{equation*}%
Then we get the estimate (\ref{enca}).
\end{proof}

\begin{remark}
(i) In case $u$ is defined in $\mathbb{R}^{N}\backslash \left\{ 0\right\} $
and $l=0,$ we obtain the estimates%
\begin{equation*}
C_{1}\left\vert x\right\vert ^{\frac{m-N}{m-1}}\leq u(x)\leq C_{2}\left\vert
x\right\vert ^{\frac{1}{\sigma }\frac{m-N}{m-1}}.
\end{equation*}%
It would be interesting to improve the estimate from above.\smallskip

\noindent (ii) If $u$ is defined in $\mathbb{R}^{N}\backslash B_{r_{0}}$ and if $\mu $
is nonincreasing, we have proved that $u$ has a limit $l\geq 0$ as $%
\left\vert x\right\vert \rightarrow \infty .$ If $\mu $ is nondecreasing, we
only obtain that $\mu (r)=\inf_{\left\vert x\right\vert =r}u(x)$ has a limit $l,$
and $\sup_{\left\vert x\right\vert =r}u(x)$ has a limit $\lambda \geq l.$
Indeed the function $w$ is $m$-subharmonic positive and bounded, so the
function $r\mapsto \sup_{\left\vert x\right\vert =r}w=(\sup_{\left\vert
x\right\vert =r}u)^{\sigma }$ is also monotone for large $r$ and has a limit 
$\lambda ^{\sigma }.$ We have $w=u^{\sigma }\leq \lambda ^{\sigma },$ so $%
\sup_{\left\vert x\right\vert =r}u$ is also nondecreasing. But we cannot
prove that $\lambda =l.$
\end{remark}

\subsection{Behaviour near an isolated singularity\label{isol}}

In this section we study the behaviour of solutions with an isolated singularity at the
origin.\medskip

\noindent {\bf Proof of Theorem \protect\ref{ball}}
Let $u$ be a nonnegative solution $u$ of (\ref{eqm}) in $B_{r_{0}}\backslash \left\{
0\right\} .$ We apply directly the Bernstein method to $u:$ we obtain by
Lemma \ref{Bern} with $b=1,$ and then $s=p.$ Setting $\xi =\left\vert \nabla
u\right\vert ^{2},$ we get 
\begin{equation*}
\frac{1}{2}\mathcal{A}(\xi )+C_{1}u^{2p}z^{q+2-m}\leq C_{2}\frac{\xi ^{2}}{%
u^{2}}+C_{3}\frac{\left\vert \nabla \xi \right\vert ^{2}}{\xi }.
\end{equation*}%
By the strong maximum principle, there exists a constant $a_{0}>0$
depending on $r_{0}$ and $N,p,q,$ such that $u\geq a_{0}$ in $%
B_{\frac {r_0}2}\backslash \left\{ 0\right\}.$ Therefore, there holds%
\begin{equation*}
\frac{1}{2}\mathcal{A}(\xi )+C_{1}^{2p}a_{0}^{2p}z^{q+2-m}\leq C_{2}\frac{%
\xi ^{2}}{a_{0}^{2}}+C_{3}\frac{\left\vert \nabla \xi \right\vert ^{2}}{\xi },
\end{equation*}%
 in $B_{\frac{r_{0}}{2}}\backslash \left\{ 0\right\}$. Then from Lemma \ref{oss}, we deduce the inequality
\begin{equation*}
z(x_{0})\leq c\left(\left(\frac{1}{a_{0}^{2p}\rho ^{2}}\right)^{\frac{1}{q+1-m}}+\left(\frac{1}{%
a_{0}^{2(p+1)}}\right)^{\frac{1}{q+2-m}}\right)\leq c_{0}^{2}\left(\frac{1}{\rho ^{\frac{2}{%
q+1-m}}}+1\right),
\end{equation*}%
for any ball $\overline{B}_\rho(x_{0})
\subset B_{\frac{r_{0}}{2}}\backslash \left\{ 0\right\}$, with $
c=c(N,p,q,m)$ and $c_{0}^{2}=c\left(a_{0}^{-\frac{p}{q+1-m}}+a_{0}^{-\frac{p+1}{q+2-m}}\right).$
Hence for any $x\in \overline{B}_{R}\backslash \left\{ 0\right\} ,$ with $%
R\leq \min (1,\frac{r_{0}}{8})$, 
\begin{equation*}
\left\vert \nabla u(x)\right\vert \leq \frac{2c_{0}}{\left\vert x\right\vert
^{\frac{1}{q+1-m}}}\leq \frac{2c_{0}}{\left\vert x\right\vert ^{\frac{1}{%
q+1-m}}}.
\end{equation*}%
As a consequence, considering any $x_{R}$ such that $x,x^{\prime }\in 
\overline{B}_{\frac{R}{2}}\backslash \left\{ 0\right\} $, there holds 
\begin{equation*}
\left\vert u(x^{\prime })-u(x)\right\vert \leq 2c_{0}R^{\frac{q-m}{q+1-m}}.
\end{equation*}%
Since $q>m,$ $u$ is bounded near $0.$ Then, with constants $C>0$ depending on $%
a_{0},$ 
\begin{equation*}
-\Delta _{m}u=f\leq C\left\vert \nabla u\right\vert ^{q}\leq C\left\vert
x\right\vert ^{-\frac{q}{q+1-m}}.
\end{equation*}%
Then $f\in L_{loc}^{\frac{N}{m}+\varepsilon }(B_{\frac{R}{4}}),$ since $N-%
\frac{N}{m}\frac{q}{q+1-m}=\frac{N(m-1)(q-m)}{m(q+1-m)}>0.$ Thus from \cite%
{Se} $u$ can be extended as a continuous function, solution of the equation
in the sense of distributions. Then we deduce that for any $x\in \overline{B}%
_{\frac{R}{2}}\backslash \left\{ 0\right\}$, 
\begin{equation*}
\left\vert u(0)-u(x)\right\vert \leq c_{0}\left\vert x\right\vert ^{\frac{q-m%
}{q+1-m}}.
\end{equation*}%
Moreover, replacing $r_{0}$ by $\rho >0$ small enough such that $u(x)\geq 
\frac{u(0)}{2}$ in \ $B_{\rho },$ then $a_{0}\geq \frac{u(0)}{2},$ hence $%
c_{0}\leq C(N,p,q,m,u(0)),$ and for $\left\vert x\right\vert \leq \min (1,%
\frac{\rho }{8}),$ we infer
\begin{equation*}
\left\vert u(0)-u(x)\right\vert \leq C\left\vert x\right\vert ^{\frac{q-m}{%
q+1-m}}.
\end{equation*}

Next  assume that $u$ is defined in $\mathbb{R}^{N}\backslash \left\{
0\right\} $ and is not constant. Then $u$ is bounded, since it is bounded
near $0$ and $\infty .$ Then $r\mapsto \mu (r)=\inf_{\left\vert x\right\vert
=r}u$ is nonincreasing, thus $\mu (r)\leq u(0)$: indeed $\forall
\varepsilon >0,$ we have $\mu (\left\vert x\right\vert )\leq u(x)\leq
u(0)+\varepsilon $ for any $\left\vert x\right\vert \leq r_{\varepsilon },$
then from the monotone decreasingness, $\mu (r)\leq u(0)+\varepsilon $ for any $r>0$%
. From Theorem \ref{exte} $\lim_{\left\vert x\right\vert \rightarrow
\infty }u=l=\lim_{r\rightarrow \infty }\mu .$ and then necessarily $l\leq
u(0).$ Suppose that there exists $x\neq 0$ such that $u(x)>u(0);$ then $u$
has a maximum in $\mathbb{R}^{N}\backslash \left\{ 0\right\} ,$ but $%
\left\vert \nabla u\right\vert $ cannot vanish in $\mathbb{R}^{N}\backslash
\left\{ 0\right\} $ by Theorem \ref{exte}, so we get a contradiction.

\section{Radial case\label{radi}}

If $u$ is a positive radial solution  of (\ref{eqm}%
), and if we denote for simplicity $u(r)=u(x)$ with $r=|x|$, then $u$ satisfies the following o.d.e. 
\begin{equation}
(\left\vert u^{\prime }\right\vert ^{m-2}u^{\prime })^{\prime }+\frac{N-1}{r}%
\left\vert u^{\prime }\right\vert ^{m-2}u^{\prime }+u^{p}\left\vert
u^{\prime }\right\vert ^{q}=r^{1-N}(r^{N-1}\left\vert u^{\prime }\right\vert
^{m-2}u^{\prime })^{\prime }+u^{p}\left\vert u^{\prime }\right\vert ^{q}=0.  \label{eqr}
\end{equation}
We begin with a simple observation about the set of zeros of $u^{\prime }.$ We have
shown above that any solution of the exterior problem is either constant, or
its gradient does not vanish. In the radial case, the proof is elementary:

\begin{proposition}
\label{van} Assume $q>m-1,$ $p\geq 0$. Then any nonnegative radial solution of (\ref{eqr}) on a segment 
$\left[ r_{1},r_{2}\right] \subset (0,\infty)$ is constant, or strictly monotone. 
\end{proposition}

\noindent\begin{proof}By the strong maximum principle \cite{Vaz}, we can assume that $u>0$ on $(r_{1},r_{2})$. The function 
$$r\mapsto W(r):=r^{N-1}\left\vert
u^{\prime }(r)\right\vert ^{m-2}u^{\prime }(r)$$
 is nonincreasing. Suppose that $%
u^{\prime }$ has two zeros $\rho _{1}$ and $\rho _{2}$ in $(r_{1},r_{2})$, then by integrating $W'$ and using the equation, we deduce that $u^{\prime }\equiv 0$ on $[\rho_{1},\rho _{2}]$, hence $u$ is constant therein, therefore we can assume that $[\rho_{1},\rho _{2}]$ is the maximal subinterval of $\left[ r_{1},r_{2}\right]$ where $u'$ vanishes.  If $[r_{1},\rho _{1}]\neq \left[ r_{1},r_{2}\right]$, for example 
$r_1<\rho_1$, then $u'>0$ on $(r_1,\rho_1)$ where $u'(r)=(r^{1-N}W(r))^{\frac{1}{m-1}}$. By (\ref{eqr}), 
\begin{equation}
\frac{m-1}{m-1-q}\left(W^{\frac{m-1-q}{m-1}}\right)^{\prime }=\left\vert W\right\vert ^{-\frac{q}{m-1}}W^{\prime }=-r^{N-1-(N-1)\frac{q}{%
m-1}}u^{p},  \label{tric}
\end{equation}
on $(r_1,\rho_1)$ and $\lim_{r\to\rho_1}u^{\prime }r)=0$. Since $m-1-q<$ this implies $\lim_{r\to\rho_1}W^{\frac{m-1-p}{m-1}}r)=\infty$, a contradiction since $u$ is bounded on $[r_1,\rho_1]$. We proceed similarly if $r_1=\rho_1$ but $\rho_2<r_2$ or if $\rho_1=\rho_2$. Hence either $u$ is constant or it is strictly monotone.
\end{proof}\medskip

Next we make a complete description of the radial solutions for $p\geq 0,$ $%
q>m.$

\subsection{The case $p=0$}

This case $p=0$ of the Hamilton-Jacobi equation is well known, since
equation (\ref{tric}) can be directly integrated, so the solutions are
explicit, and are a Ariadne's thread for studying  the case $p>0$. We find different types
of nonconstant solutions according to the sign of $u^{\prime }:$ 
\begin{equation*}
\left\{ 
\begin{array}{c}
u^{\prime }=r^{\frac{1-N}{m-1}}(C_{1}-a_{m,q}^{-1}r^{-a_{m,q})})^{-\frac{1}{%
q-m+1}} \\ 
u^{\prime }=-r^{\frac{1-N}{m-1}}(C_{2}+a_{m,q}^{-1}r^{-a_{m,q})})^{-\frac{1}{%
q-m+1}}%
\end{array}%
\right.
\end{equation*}%
where $a_{m,q}=\frac{(N-1)q-N(m-1)}{m-1}>0$ since $q>\frac{N(m-1)}{N-1}>m-1;$
and the value of $u$ follows by integration, with the requirement that $u>0$. The
solutions such that $C_{1}>0$ satisfy $\lim_{r\rightarrow 0}r^{\frac{1}{q-m+1%
}}u^{\prime }(r)=-a_{m,q}^{\frac{1}{q-m+1}},$ then $\lim_{r\rightarrow
0}u(r)=u_{0}>0$ , since $q>m.$ The conclusions of theorem \ref%
{radial} follow in that case.

\subsection{The case $p>0$}

Equation (\ref{eqr}) can be reduced to an autonomous system, since it is
invariant by the transformation $u\mapsto T_{\lambda }u$ ($\lambda>0$) given by 
\begin{equation}
T_{\lambda }u(x)=\lambda ^{-\frac{q-m}{p+q+1-m}}u(\lambda x).  \label{sca}
\end{equation}%
Here we perform a change of unknown, introduced in \cite{BiGi}, which
consists in a differentiation of the equation, as in the Bernstein technique. We
set 
\begin{equation}
X(t)=-r\frac{u^{\prime }(r)}{u(r)},\qquad Z(t)=-ru^{p}\left\vert u^{\prime
}\right\vert ^{q-m}u^{\prime },\qquad t=\ln r,  \label{xz}
\end{equation}%
and obtain the following quadratic system of Kolmogorov type, valid any point $t$
where $u^{\prime }(t)\neq 0,$ and any reals $m,p,q,$ 
\begin{equation}
\left\{ 
\begin{array}{ccc}
X_{t} & = & X(X-\frac{N-m}{m-1}+\frac{Z}{m-1}) \\ [2mm]
Z_{t} & = & Z(N-\frac{N-1}{m-1}q-pX+\frac{q+1-m}{m-1}Z),%
\end{array}%
\right.  \label{XZ}
\end{equation}%
in the region $Q=\left\{ (X,Z)\in \mathbb{R}^{2}:XZ>0\right\} $. Note that
the trajectories $X=0$ and $Z=0$ are not admissible in our study. Since $%
p+q\neq m-1,$ we can recover $u$ and $u^{\prime }$ by 
\begin{equation}
u=(r^{q-m}\left\vert Z\right\vert \left\vert X\right\vert ^{m-1-q})^{\frac{1%
}{p+q-m+1}},\quad u^{\prime }=-\frac{Xu}{r}=(r^{-(p+1)}\left\vert
Z\right\vert \left\vert X\right\vert ^{p})^{\frac{1}{p+q-m+1}}\text{sign}(-X).
\label{auto}
\end{equation}%
The fixed points of the system in $\overline{Q}$ are 
\begin{equation*}
N_{0}=(0,a_{m,q})=(0,\frac{(N-1)q-N(m-1)}{q+1-m}),\qquad O=(0,0),\qquad
A_{0}=(\frac{N-m}{m-1},0).
\end{equation*}%
We begin by a local study of the different points and the correponding
results for the solutions of (\ref{eqr}):

\begin{lemma}
(i) The point $N_{0}$ is a source, with eigenvalues $0<\lambda _{1}=\frac{q-m%
}{q+1-m}<\lambda _{2}=\frac{(N-1)q-N(m-1)}{m-1}$ and eigenvectors $%
v_{1}=(1,c_{m,q})$ with $c_{m,q}>0$ and $v_{2}=(0,1).$ Then there exist
infinitely many singular decreasing solutions $u$ of (\ref{eqr}) defined near $%
0, $ satisfying (\ref{ima}).\smallskip

\noindent (ii) The point $O$ is a sink, with eigenvalues $0>$ $\xi _{1}=-\frac{N-m}{m-1%
}>\xi _{2}=N-\frac{N-1}{m-1}q,$ and eigenvectors $u_{1}=(1,0)$ and $%
u_{2}=(0,1)$. Then there exist infinitely many solutions $u$ of (\ref{eqm})
defined for large $r$ and either increasing or decreasing near $\infty ,$
satisfying (\ref{ime}) and (\ref{imo})\smallskip

\noindent (iii) The point $A_{0}$ is a saddle point, with eigenvalues $\mu _{1}=-\frac{%
(N-m)p+(N-1)q-(m-1)N}{m-1}<0<\mu _{2}=\frac{N-m}{m-1}$ and eigenvectors $%
w_{1}=(1,-d_{m,q})$ with $d_{m,q}>0$ and $w_{2}=(1,0).$ Then for any $c>0$
there exists a unique solution $u$ of (\ref{eqr}), defined at least for
large $r,$ such that $\lim_{r\rightarrow \infty }r^{\frac{N-m}{m-1}}u=c>0.$
\end{lemma}
\noindent\begin{proof}
(i) We perform the linearization at $N_{0}:$ setting $Z=a_{m,q}+\overline{Z},$
we get, with $X>0$ and $Z>0,$ 
\begin{equation*}
\left\{ 
\begin{array}{ccc}
X_{t} & = & \frac{q-m}{q+1-m}X \\ 
\overline{Z}_{t} & = & a_{m,q}(-pX+\frac{q+1-m}{m-1}\overline{Z}),%
\end{array}%
\right.
\end{equation*}%
which gives the eigenvalues $\lambda _{1},\lambda _{2}$ and their
respective eigenvectors, with the value of  $c_{m,q}$
\begin{equation*}
c_{m,q}=\frac{p(N-1)q-N)(m-1)}{(N-1)q^{2}-2N(m-1)q+(m-1)(N(m-1)+m)}.
\end{equation*}%
So $N_{0}$ is a source; the particular trajectory $X=0$ associated to $%
\lambda _{2}$ is not admissible. There exists an infinity of trajectories
starting from $N_{0}$ as $t\rightarrow -\infty $, associated to the
eigenvalue $\lambda _{1};$ the solutions $(X,Z)$ satisfy $X>0,$ $%
\lim_{t\rightarrow -\infty }e^{-\frac{q-m}{q+1-m}t}X=C_{0},$ where $C_{0}>0$
is arbitrary and $\lim_{t\rightarrow -\infty }Z=a_{m,q}$; then from (\ref%
{auto}) and the definition of $Z$ , there exist  infinitely many decreasing singular
solutions $u$ of (\ref{eqr}) defined near $0,$ satisfying (\ref%
{ima}).\smallskip

\noindent (ii) The linearisation at $O$ gives the system 
\begin{equation*}
\left\{ 
\begin{array}{ccc}
X_{t} & = & -\frac{N-m}{m-1}X \\ 
Z_{t} & = & (N-\frac{N-1}{m-1}q)Z,%
\end{array}%
\right.
\end{equation*}%
with admits the eigenvalues $\xi _{1},$ $\xi _{2}.$ So $O$ is a sink, two
particular trajectories are the axis $X=0$ and $Z=0$ which not admissible. There
is an infinity of trajectories converging to $O$ as $t\rightarrow \infty ,$
tangent to the axis $Z=0,$ associated to the eigenvalue $\xi _{1},$ with
either $X,Z>0$, or $X,Z<0$ . They satisfy 
\begin{equation}
X\sim _{t\rightarrow \infty }C_{1}e^{-\frac{N-m}{m-1}t},\quad Z=\sim
_{t\rightarrow \infty }C_{2}e^{(N-\frac{N-1}{m-1}q)t},\text{ with }%
C_{1},C_{2}>0.  \label{c12}
\end{equation}%
The corresponding solutions $u$ of (\ref{eqr}) are defined for large $r,$
and either decreasing or increasing; from (\ref{auto}), we obtain $%
\lim_{r\rightarrow \infty }u=(C_{1}^{m-1-q}C_{2})^{\frac{1}{p+q-m+1}}=l>0$
and $\lim_{r\rightarrow \infty }r^{\frac{N-1}{m-1}}u^{\prime }=-lC_{1},$
thus $\lim_{r\rightarrow \infty }r^{\frac{N-m}{m-1}}(u-l)=-lC_{1}.$ Thus (%
\ref{ime}) and (\ref{imo}) follow. The uniqueness property follows from the
uniqueness of a trajectory satisfying (\ref{c12}) for given $C_{1},C_{2},$
see also Remark \ref{cha} below.\smallskip

\noindent (iii) Linearisation at $A_{0}$ $:$ setting $X=\frac{N-m}{m-1}+\overline{X},$
we get%
\begin{equation*}
\left\{ 
\begin{array}{ccc}
\overline{X}_{t} & = & \frac{N-m}{m-1}(\overline{X}+\frac{Z}{m-1}) \\ 
Z_{t} & = & -\frac{(N-m)p+(N-1)q-(m-1)N}{m-1}Z,%
\end{array}%
\right.
\end{equation*}%
which admits the eigenvalues $\mu _{1}<0<\mu _{2}$ and the eigenvectors, with 
\begin{equation*}
d_{m,q}=\frac{m-1}{N-m}(N-m+\left\vert \mu _{1}\right\vert (m-1)
\end{equation*}
It is a saddle point. The trajectory $X=0$ associated to $\mu _{2}$ is not
admissible. Then a unique trajectory $\mathcal{T}_{A_{0}}$ converging to $%
A_{0}$ as $t\rightarrow \infty .$ By the scaling (\ref{sca}), we deduce the
uniqueness property for $u$.
\end{proof}

Next we a complete description of the local and global solutions in the
phase-plane leading to the conclusions of Theorem \ref{radial}:\medskip

\noindent\begin{proof}[Proof of Theorem \protect\ref{radial} when $p>0$]
\label{mark} We consider the sets 
\begin{equation*}
\mathcal{L}_{X}\mathcal{=}\left\{ (X,Z)\in Q:X_{t}=0\right\} =\left\{
(X,Z)\in Q:X-\frac{N-m}{m-1}+\frac{Z}{m-1}=0\right\},
\end{equation*}%
\begin{equation*}
\mathcal{L}_{Z}\mathcal{=}\left\{ (X,Z)\in Q:Z_{t}=0\right\} =\left\{
(X,X)\in Q:N-\frac{N-1}{m-1}q-pX+\frac{q+1-m}{m-1}Z\right\}.
\end{equation*}%
The straight line $\mathcal{L}_{X}$ has an extremity at $A_{0},$ with slope $-(m-1),$
and the slope of $\mathcal{T}_{A_{0}}$ is $-d_{m,q}<-(m-1)$, so $\mathcal{T}%
_{A_{0}}$ is above $\mathcal{L}_{X}$ near $t=\infty .$ The line $\mathcal{L}%
_{X}$ has an extremity at $N_{0}$ , with slope $\frac{p(m-1)}{q+1-m},$ is
located above $\mathcal{L}_{X}$ for $X>0.$ The trajectories issued from $%
N_{0}$ have the slope $c_{m,q},$ and we check that it is greater than $\frac{%
p(m-1)}{q+1-m}$ because $q>1-m;$ so they start above $\mathcal{L}_{Z}.$

(i) The trajectory $\mathcal{T}_{A_{0}}$ stays in the region $\mathcal{R=}%
\left\{ 0<X<\frac{N-m}{m-1};X-\frac{N-m}{m-1}+\frac{Z}{m-1}>0\right\} $
which is negatively invariant. Then $X_{t}$ stays positive, thus $X$ is
increasing, hence bounded. Either $\mathcal{T}_{A_{0}}$ stays under $%
\mathcal{L}_{Z},$ then $Z_{t}<0$ and $Z$ is bounded, thus $\mathcal{T}%
_{A_{0}}$ converges to $N_{0},$ or it crosses the line $\mathcal{L}_{Z}$ at
time $t_{0},$ and for $t<t_{0}$ there holds $Z_{t}>0$ so that $Z$ stays
bounded, and $\mathcal{T}_{A_{0}}$ still converges to $N_{0}$; in fact the
second eventuality holds, because of the slope of the eigenvector at $N_{0}.$
So the trajectory, $\mathcal{T}_{A_{0}}$ joins $N_{0}$ to $A_{0}$. By
scaling, for any $u_{0}>0$ there exists a unique solution $u$ defined in $%
(0,\infty )$ satisfying (\ref{imu}).

(ii) All the trajectories with one point in the bounded invariant region $%
\mathcal{R}^{\prime }$ delimitated by the axis $X=0,Z=0$ and $\mathcal{T}%
_{A_{0}},$ join $N_{0}$ to $O,$ and the corresponding solutions $u$ are
positive on $(0,\infty )$, decreasing, and satisfy (\ref{ima}). The
trajectories with one point in the region $\mathcal{R}^{\prime \prime
}\subset Q$ above $\mathcal{T}_{A_{0}}$ converge to $N_{0}$ as $t\rightarrow
-\infty ,$ and satisfy $X_{t}>0,$ since $\mathcal{T}_{A_{0}}$ is above $%
\mathcal{L}_{X},$ and cannot be bounded, since there is no fixed point in
this region. They can be of two types:

$\bullet $ Either they cross $\mathcal{L}_{Z},$ then after crossing $Z$ is
decreasing, necessarily to $0;$ then from (\ref{auto}), $u$ is defined in a
maximal interval $(0,\rho )$ with $u(\rho )=0.$ Such solutions exists
because by any point on $\mathcal{L}_{Z}$ passes a trajectory.

$\bullet $ Or they stay above $\mathcal{L}_{Z},$ thus $Z$ increases to $%
\infty ;$ in this case from (\ref{auto}) $u$ is defined in a maximal
interval $(0,\rho )$ with $\lim_{r\rightarrow \rho }$ $u^{\prime }=-\infty ;$
Let us show the existence of such solutions: For given $c>0,$ we define%
\begin{equation*}
\mathcal{L}_{c}=\left\{ (X,Z)\in Q:X>0,Z=cX+a_{m,q}\right\}.
\end{equation*}%
We compute the field on this line, and show that it is entering the region
above $\mathcal{L}_{c}$ for $c$ large enough: indeed we obtain, 
\begin{eqnarray*}
\frac{Z_{t}}{X_{t}}-c &=&\frac{Z(\frac{q+1-m}{m-1}(Z-a_{m,q})-pX)}{X(X-\frac{%
N-m}{m-1}+\frac{cX+a_{m,q}}{m-1})}-c \\
&=&\frac{Z(c\frac{q+1-m}{m-1}-p)}{X-\frac{N-m}{m-1}+\frac{cX+a_{m,q}}{m-1}}%
-c>\frac{(c(q+1-m)-p(m-1))(cX+a_{m,q})}{(m-1+c)X+a_{m,q}-(N-m)}-c,
\end{eqnarray*}%
and 
\begin{equation*}\begin{array}{lll}
(c(q+1-m)-p(m-1))(cX+a_{m,q})-c((m-1+c)X+a_{m,q}-(N-m)) \\[2mm]
\;=cX(c(q-m)-(p+1)(m-1)-(p+a_{m,q}(m-1))+c(q-m)a_{m,q}+N-m)-p(m-1)a_{m,q}
\end{array}\end{equation*}%
is positive for large $c,$ since $q-m>0.$ All the solutions with one point
above $\mathcal{L}_{c}$ stay in this region, so above $\mathcal{L}_{Z},$
which proves the existence.

(iii) All the trajectories with one point in $\left\{ (X,Z)\in Q:X<0\right\} 
$ satisfy $X_{t}>0$ from (\ref{XZ}). Then $X$ increases necessarily up to $0,$
and then $Z_{t}>0$ for large $t,$ thus $(X,Z)$ converges to $O,$ and $u$ is
defined for $r$ large enough, increasing and $\lim_{r\rightarrow \infty
}u=l>0.$

$\bullet $ Either they cross $\mathcal{L}_{Z},$ then before crossing $Z$ is
decreasing, necessarily to $0;$ then from (\ref{auto}), $u$ is defined in a
maximal interval $(0,\rho )$ with $u(\rho )=0.$ Such solutions exist as
above.

$\bullet $ Or they stay under $\mathcal{L}_{Z},$ thus $X$ and $Z$ decrease
to $-\infty ;$ in this case from (\ref{auto}) $u$ is defined in a maximal
interval $(0,\rho )$ with $\lim_{r\rightarrow \rho }$ $u^{\prime }=-\infty .$
Let us show their existence: for given $k>0$ we compute the field on the
line $\mathcal{L}^{k}=\left\{ (X,Z)\in Q:X<0,Z=kX\right\} .$ On this line $%
X_{t}>0$ and 
\begin{equation*}\begin{array}{lll}
Z_{t}-kX_{t} =Z(N-\frac{N-1}{m-1}q-pX+\frac{q+1-m}{m-1}Z-X+\frac{N-m}{m-1}-%
\frac{Z}{m-1}) \\[2mm]\phantom{Z_{t}-kX_{t} }
=Z((\frac{q-m}{m-1}-\frac{p+1}{k})Z-\frac{N-1}{m-1}(q-m))=Z^{2}(\frac{q-m}{%
m-1}-\frac{p+1}{k})+\frac{N-1}{m-1}(q-m)\left\vert Z\right\vert,
\end{array}\end{equation*}%
is positive for large $k.$ The region below $\mathcal{L}^{k}$ is therefore
negatively invariant, then the existence is folllows. This conclude the
proof.
\end{proof}

\begin{remark}
\label{cha}The change of variable $u(r)=\tilde{u}(s),s=r^{\frac{m-N}{%
m-1}}$, introduced in \cite{GuVe}, and also used in \cite{BuGMQu} in case $m=2<N,$ leads to the equation 
\begin{equation}
(\left\vert \tilde{u}_{s}\right\vert _{s}^{m-2}\tilde{u})_{s}+\left(\frac{m-N}{m-1%
}\right)^{q-m}s^{\frac{N-1}{N-m}(q-m)}\tilde{u}^{p}\left\vert \tilde{u}%
_{s}\right\vert ^{q}=0.  \label{ori}
\end{equation}%
Hence if $u$ is not constant $\tilde{u}_{s}$ does
not vanish, from Remark \ref{van}, and (\ref{ori}) is equivalent to 
\begin{equation}
(m-1)\tilde{u}_{ss}+\left(\frac{m-N}{m-1}\right)^{q-m}s^{\frac{N-1}{N-m}(q-m)}\tilde{u}%
^{p}\left\vert \tilde{u}_{s}\right\vert ^{q-m+2}=0.  \label{ord}
\end{equation}%
In particular we find again the existence and uniqueness of local solutions
near $\infty ,$ satisfying (\ref{imo}) for given $l\geq 0$ and $c\neq 0$ $%
(c>0$ if $l=0);$ indeed the problem reduces to the equation (\ref{ord}) with the initial conditions $\tilde{u}(0)=l$ and $\tilde{u}_{s}(0)=c.$
\end{remark}

\section{The case $p<0$\label{exten}}
\noindent\begin{proof}[Proof of Theorem \protect\ref{pneg}]
We still consider $u=v^{b},$ with $b>0:$ we recall that from (\ref{eqv}) (%
\ref{cos}) 
\begin{equation*}
-\Delta _{m}v=(b-1)(m-1)\frac{\left\vert \nabla v\right\vert ^{m}}{v}%
+b^{q-m+1}v^{s}\left\vert \nabla v\right\vert ^{q},
\end{equation*}%
with $s=1-q+m+b(p+q-m+1)$. Next we take 
\begin{equation*}
b=\frac{q+1-m}{p+q-m+1},
\end{equation*}%
thus here $b\geq 1$ and $s=0,$ so, 
\begin{equation}
-\Delta _{m}v=(b-1)(m-1)\frac{\left\vert \nabla v\right\vert ^{m}}{v}%
+b^{q-m+1}\left\vert \nabla v\right\vert ^{q},  \label{sig}
\end{equation}%
where the two terms have the same sign. Then $z=\left\vert \nabla
v\right\vert ^{2}$ satisfies 
\begin{equation*}
\mathcal{A}(v)=-\Delta v-\frac{m-2}{2}\frac{<\nabla z,\nabla v>}{z}%
=(b-1)(m-1)\frac{z}{v}+b^{q-m+1}z^{\frac{q+2-m}{2}}.
\end{equation*}%
Setting $,$ we get from (\ref{estiv})%
\begin{eqnarray*}
&&-\frac{1}{2}\mathcal{A}(z)+\frac{1}{N}(\Delta v)^{2}+(b-1)(m-1)\frac{z^{2}%
}{v^{2}} \\
&\leq &(b-1)(m-1)(\frac{<\nabla z,\nabla v>}{v}+\frac{q+2-m}{2}b^{q-m+1}z^{%
\frac{q-m}{2}}<\nabla z,\nabla v>,
\end{eqnarray*}%
where now the term in $\frac{z^{2}}{v^{2}}$ has a positive coefficient.
Since $b\geq 1,$ we get an estimate of the form 
\begin{equation*}
\mathcal{A}(z)+C_{1}z^{q+2-m}\leq C_{3}\frac{\left\vert \nabla z\right\vert
^{2}}{z}.
\end{equation*}%
Since $q+2-m>1,$ we deduce the estimate in any ball $B_\rho (x_{0})$, 
\begin{equation*}
\left\vert \nabla v(x_{0})\right\vert \leq C\left(\frac{1}{\rho }\right)^{\frac{1}{q+1-m%
}},
\end{equation*}%
from Lemma \ref{oss}, where $C$ is a universal constant, which leads to the
conclusions.\end{proof}

\end{document}